\documentclass[11pt,a4paper]{article}

\usepackage{amssymb}
\usepackage[usenames]{color}
\usepackage{theorem}
\usepackage[utf8]{inputenc}
\usepackage{enumerate}
\usepackage{amsmath}
\usepackage{graphicx} 
\graphicspath{ {images/} }
\usepackage{float}

\numberwithin{equation}{section}

\newtheorem{theorem}{Theorem}[section]
\newtheorem{proposition}[theorem]{Proposition}
\newtheorem{lemma}[theorem]{Lemma}
\newtheorem{corollary}[theorem]{Corollary}
\newtheorem{definition}[theorem]{Definition}
\newtheorem{remark}[theorem]{Remark}

\newtheorem{convention}[theorem]{Convention}

\newenvironment{proof}{\noindent\textbf{Proof}}
                      {$\Box$\vskip\theorempostskipamount}

\begin{document}

\title{\textbf{About the cyclically reduced product of words}}

\author{Carmelo Vaccaro}

\date{}

\maketitle

\begin{abstract} The cyclically reduced product of two words is the cyclically reduced form of the concatenation of the two words. While the reduced form of such a concatenation (which is the product of the free group) verifies many basic properties like for example associativity, the same is not true for the cyclically reduced product which has been very little studied in the literature.

Recently S. V. Ivanov has proved that the Andrews-Curtis conjecture (stated in 1965 and still not solved) is equivalent to a formulation where the reduced product is replaced by the cyclically reduced product (and the conjugations replaced by cyclic permutations).

In this paper we study properties of the cyclically reduced product $*$ and of the set of cyclically reduced words $\hat{\mathcal{F}}(X)$ equipped with $*$. In particular we find that even if $*$ is not commutative nor verifies the Latin square property, generalized versions of these properties hold true.

We also show that $\hat{\mathcal{F}}(X)$ equipped with $*$ and with cyclic permutations enjoys similar properties as the free group equipped with the reduced product and conjugations.
 \end{abstract}

\smallskip \smallskip


\smallskip \smallskip


\section*{Introduction}

Let $X$ be a set of letters, let $X^{-1}$ be the set of inverses of elements of $X$ and let $\mathcal{M}(X \cup X^{-1})$ be the free monoid on $X \cup X^{-1}$. The elements of $\mathcal{M}(X \cup X^{-1})$ are the non-necessarily reduced words on $X$. We denote $\mathcal{F}(X)$ the free group on $X$ and we consider it as the subset of $\mathcal{M}(X \cup X^{-1})$ consisting of reduced words. We denote $\hat{\mathcal{F}}(X)$ the set of cyclically reduced words on $X$.

Let $\langle X | R \rangle$ be a presentation for a group $G$, with $X$ the set of generators and $R$ that of basic relators; then $G$ is isomorphic to $\mathcal{F}(X) / \mathcal{N}(R)$, where $\mathcal{N}(R)$ is the normal subgroup of $\mathcal{F}(X)$ normally generated by $R$. In particular $\mathcal{N}(R)$ is the set of all relators and is the smallest subset of $\mathcal{F}(X)$ containing $R$ and closed with respect to the reduced product and to the conjugations.

Given $v, w \in \mathcal{M}(X \cup X^{-1})$, for example $v = t x y$ and $w = y^{-1} z t^{-1}$ with $t, x, y, z \in X$, we can define three different products of $v$ by $w$:

\begin{enumerate}
	\item the \textit{concatenation}, $vw := t x y y^{-1} z t^{-1}$;
	
	\item the \textit{reduced product}, $v \cdot w := t x z t^{-1}$, namely the reduced form of the concatenation;
	
	\item the \textit{cyclically reduced product}, $u *v := x z$,	namely the cyclically reduced form of the concatenation.
\end{enumerate}

While the first two products (which correspond respectively to the products in the free monoid and in the free group) are very well studied, the third one has been almost completely ignored in the literature. This is maybe due to the fact that the cyclically reduced product does not verify important properties like for example associativity, so the structure of $\hat{\mathcal{F}}(X)$ is not as nice as that of $\mathcal{F}(X)$.

\smallskip

Despite the fact that the cyclically reduced form of a word plays a crucial role in the solution of the conjugacy problem for the free group, the cyclically reduced product has lacked applications to important research problems and has been considered only an exotic concept until not long ago: at our knowledge, before 2006 the only papers dealing with it were \cite{Rourke} and \cite{Scarabotti}.

But things have changed recently because in two papers of 2006 \cite{Ivanov06} and of 2018 \cite{Ivanov18} S. V. Ivanov has proved an extremely interesting result concerning the Andrews-Curtis conjecture: the conjecture (with and without stabilizations) is true if and only if in the definition of the conjecture we replace the operations of reduced product and conjugations with the cyclically reduced product and the cyclic permutations. 

The importance of this result stems from the fact that while there are infinitely many conjugates of one word, there are only finitely many cyclic permutations, thus making much easier the search of Andrews-Curtis trivializations by enumerations of relators, like for example the approaches used in \cite{BMc} or \cite{PU}.

\bigskip

\noindent \textbf{Goal of the paper.}

The goal of the present paper and of the two following ones (\cite{TwAs1} and \cite{TwAs2}) is to fill a hole in the literature concerning the study of the properties of the cyclically reduced product of words and the structure of $(\hat{\mathcal{F}}(X), *)$. In particular we show that $(\hat{\mathcal{F}}(X), *)$ is a magma\footnote{A magma is a set closed with respect to an operation} with a unique identity, without non-trivial idempotents and with each element having a unique inverse. We also show that $*$ is not associative, commutative nor verifies the Latin square property; however generalized versions of these properties hold true (see Section \ref{propHatF} for more details).

It appears that the structure of $\hat{\mathcal{F}}(X)$ becomes more interesting when we consider it equipped not only with the cyclically reduced product but also with the cyclic permutations. In that case (as it is shown also in the next section, Related works), $\hat{\mathcal{F}}(X)$ enjoys similar properties as those of the free group equipped with the reduced product and conjugations. 

In particular we prove in Theorem \ref{puzo} that for words $u$ and $v$ the cyclically reduced product $u*v$ is a cyclic permutation of $v*u$ and the identity among relations that follows from this fact is a generalization of the identity among relations that follows from the fact that in the free group $u \cdot v$ is a conjugate of $v \cdot u$. We also show in Remark \ref{vanKamp} that if $u$ and $v$ are relators of a group presentation, then the van Kampen diagrams associated with $u*v$ and $v*u$ are homeomorphic.

\bigskip

\noindent \textbf{Related works.}

The first work (at our knowledge) where the cyclically reduced product was used was a famous paper of C. Rourke of 1979 \cite{Rourke}. Here the author denoted $R_{\infty}$ the smallest subset of $\mathcal{F}(X)$ containing $R$ and closed with respect to the cyclically reduced product and to cyclic permutations. The definition of $R_{\infty}$ is like that of $\mathcal{N}(R)$ but with the cyclically reduced product and cyclic permutations instead of the reduced product and conjugations. The main result of \cite{Rourke} was that if $\langle X | R \rangle$ is a presentation of the trivial group then $R_{\infty}$ contains $X$, the set of generators of the presentation.


We observe that $R_{\infty} \subset \mathcal{N}(R)$, because the cyclically reduced product of two words is equal to a conjugate of the reduced product of these two words and because a cyclic permutation is a special case of conjugation.


This paper of Rourke is very well known because it is one of the firsts where the notion of ``pictures"\footnote{A picture is a sort of dual of a van Kampen diagram, see Ch$\ldotp$ 2 of \cite{Short}} was used, but surprisingly its main result has been almost completely ignored in the literature. The first article to rediscover it was a paper of 1999 of F. Scarabotti, where the main result of \cite{Rourke} was proved using an elementary algebraic argument instead of the advanced algebraic topology one used in \cite{Rourke}.


In 2006 S. V. Ivanov proved \cite{Ivanov06} that if for every $x \in X$ either $x$ or $x^{-1}$ occurs in at least one basic relator\footnote{We observe that the previous hypothesis is not restrictive because if $X = X_1 \cup X_2$ with the elements of $X_1$ or their inverses occurring in at least one basic relator and the elements of $X_2$ not occurring, then the group presented by $\langle X | R \rangle$ is the free product of $G_1$ times $\mathcal{F}(X_2)$, where $G_1$ is the group presented by $\langle X_1 | R \rangle$.} (i.e., in an element of $R$) then $R_{\infty}$ is the set of cyclically reduced relators, i.e., $R_{\infty} = \mathcal{N}(R) \cap \hat{\mathcal{F}}(X)$. 

In the above mentioned paper and in \cite{Ivanov18} S. V. Ivanov proved that the Andrews-Curtis conjecture is true if and only if the reduced product and the conjugations are replaced by the cyclically reduced product and the cyclic permutations. We observe that one side of this equivalence is obvious in view of the fact that $R_{\infty} \subset \mathcal{N}(R)$.

\bigskip

\noindent \textbf{Structure of the paper.}

In Section \ref{S1} we give the basic definitions and prove some elementary results about the reduced product, cyclic permutations and reversions of words.

In Section \ref{S2} we introduce and study the basic properties of cyclically reduced words and the cyclically reduced product.

In Section \ref{propHatF} we discuss the structure of $\hat{\mathcal{F}}(X)$ equipped with the cyclically reduced product.

In Section \ref{S3} we prove the main results, namely the fact that the cyclically reduced product $u*v$ is a cyclic permutation of $v*u$. This result is analogous to the well known result that the reduced product $u \cdot v$ is conjugate to $v \cdot u$. But the analogy is deeper than that, in particular the identity among relations that follows from the fact that $u*v$ is a cyclic permutation of $v*u$ is analogous to the identity that follows from the fact that $u \cdot v$ is conjugate to $v \cdot u$.

Appendix \ref{AppA} formally expresses these concepts concerning identities among relations. Finally in Appendix \ref{AppB} some technical results needed in Sections \ref{S2} and \ref{S3} are proved.

The dependencies among the different sections are the following: Section \ref{S1} and Appendix \ref{AppA} are independent from the other sections, although Appendix \ref{AppA} uses some notations defined in Section \ref{S1}. Section \ref{S2} depends on Section \ref{S1}. Appendix \ref{AppB} depends on  Sections \ref{S1} and \ref{S2}. Section \ref{S3} depends on all the others sections.

\section{Words, cyclic permutations and the reduced product} \label{S1}

Let $Y$ be a set and let us consider $\mathcal{M}(Y)$, the free monoid on $Y$. The elements of $Y$ are called \textit{letters}, those of $\mathcal{M}(Y)$ \textit{the words in $Y$}. As usual given words $v, w \in \mathcal{M}(Y)$ we will denote $vw$ the product of $v$ by $w$, which is the concatenation of the words $v$ and $w$. The word with no letters, which is the identity element of $\mathcal{M}(Y)$, is denoted 1. 

Let $v, w \in \mathcal{M}(Y)$; we say that $v$ is a \textit{subword of $w$} if there exist $p, q \in \mathcal{M}(Y)$ such that $w = p v q$. In this case we say that $v$ is a \textit{prefix of $w$} if $p = 1$ and that $v$ is a \textit{suffix of $w$} if $q = 1$.

Let $v = y_1 \dots y_n \in \mathcal{M}(Y)$ with $y_1, \dots, y_n \in Y$. The \textit{length of $v$} is defined as $|v| := n$. The \textit{reverse of $v$} is defined as the word $\underline{v} := y_n \dots y_1$, where the order of the letters is the reverse as that of $v$. 

\begin{remark} \label{reverse} \rm The reverse of the reverse of a word $v$ is $v$ itself. It is also obvious that if $v_1, v_2, \dots, v_n \in \mathcal{M}(Y)$ then the reverse of $v_1 v_2 \dots v_n$ is the word $\underline{v_n} \dots \underline{v_2} \, \underline{v_1}$. 
\end{remark}

\begin{remark} \label{commFG} \rm Two elements of a free group commute if and only if they are power of the same element, i.e., if $u, v \in \mathcal{F}(X)$ are such that $\rho(uv) = \rho(vu)$ then there exist $c \in \mathcal{F}(X)$ and $m, n \in \mathbb{Z}$ such that $u = \rho(c^m)$ and $v = \rho(c^n)$ (see Proposition I.2.17 of \cite{LS}).

In particular let $a, b, u \in \mathcal{F}(X)$ be such that $\rho(a u a^{-1}) = \rho(b u b^{-1})$. This implies that $\rho(b^{-1} a u) = \rho(u b^{-1} a)$ and since $\rho(b^{-1} a)$ and $u$ commute, there exist $c \in \mathcal{F}(X)$ and $m, n \in \mathbb{Z}$ such that $u = \rho(c^m)$ and $\rho(b^{-1} a) = \rho(c^n)$, in particular $a = \rho(b c^n)$.
\end{remark}

\begin{remark} \label{LeviLemma} \rm We have the following result, known as \textit{Levi's Lemma} (see \cite{ChofKar}, pag. 333 or \cite{Kar}, Theor. 2): \textit{let $u_1, u_2, v_1, v_2$ be words such that $u_1 u_2 = v_1 v_2$; then there exists a word $p$ such that either $u_1 = v_1 p$ and $v_2 = p u_2$ or $v_1 = u_1 p$ and $u_2 = p v_2$.} The two cases can be represented graphically in the following way,
	
\medskip
	
\hspace{2cm} 
\begin{tabular}{|c|c|}
	\hline 
	$\,\, u_1 \,$   & $u_2 \,$ \\
	\hline
\end{tabular}
	\quad  \hspace{1cm} and \hspace{1cm}
\begin{tabular}{|c|c|}
	\hline 
	$u_1$ & $\,\,u_2\,$ \\
	\hline
\end{tabular}
	
\hspace{2cm}
\begin{tabular}{|c|c|}
	\hline 
	$v_1$ & $\,\, v_2 \,\,\,$ \\
	\hline
\end{tabular}
	\quad \hspace{28.8mm}
\begin{tabular}{|c|c|}
	\hline 
	$\,\,v_1 \,\,$   & $v_2$ \\
	\hline
\end{tabular}
	
\medskip
	
\noindent and correspond to putting the bar separating $v_1$ and $v_2$ either inside $u_1$ or inside $u_2$. The case when this bar is exactly below that separating $u_1$ and $u_2$, i.e., when $u_1 = v_1$ and $u_2 = v_2$, can be considered a special case of both the cases. 
	
In general let us consider the word equation $u_1 \dots u_m = v_1 \dots v_n$, possibly with $m \neq n$. Any solution to this equation determines uniquely a way of putting $n - 1$ bars inside the $m$ spaces corresponding to $u_1, \dots, u_m$ and also a way of putting $m - 1$ bars inside the $n$ spaces corresponding to $v_1, \dots, v_n$. This is true even if some of the $u_i$ or $v_j$ are the empty word. Indeed if $u_i = 1$ or $v_j = 1$ then no bar must be contained in $u_i$ or $v_j$. 
	
We observe that a solution to the equation $u_1 \dots u_m = v_1 \dots v_n$ determines also a \textit{weak composition}\footnote{a weak composition for an integer is a composition when 0's are allowed} for $n - 1$ in $m$ parts and for $m - 1$ in $n$ parts.
	
We give the following as an example for $m = 4$ and $n = 3$:
	
\medskip

\hspace{3.5 cm}
\begin{tabular}{|c|c|c|c|}
	\hline 
	$u_1$ & $u_2$ &	$u_3$ & $\,\, u_4 \,\,$ \\
	\hline
\end{tabular}
	
\hspace{3.5 cm}
\begin{tabular}{|c|c|c|}
	\hline 
	$\,\, v_1 \,\,$ & $\,\,\,\,\,\,\, v_2 \,\,\,\,\,\,\,$ & $v_3$\\
	\hline
\end{tabular}

\medskip
	
\noindent Here we can say that there exist words $a, b, c$ such that $v_1 = u_1 a$, $u_2 = a b$, $v_2 = b u_3 c$ and $u_4 = c v_3$. This solution determines the weak compositions $(0, 1, 0, 1)$ for 2 and $(1, 2, 0)$ for 3, which are obtained by counting the number of bars inside each $u_i$ and each $v_j$ respectively.
\end{remark}

\begin{definition} \label{cycPermDef} \rm Let $w_1$ and $w_2$ be words and let $w:=w_1 w_2$. The word $w_2 w_1$ is called a \textit{cyclic permutation of $w$}. Given two words $u$ and $v$ the relationship ``$u$ is a cyclic permutation of $v$" is an equivalence that we denote $u \sim v$.
\end{definition}


\begin{remark} \label{reversesim} \rm It is obvious that if $u$ and $v$ are words and if $u \sim v$ then $\underline{u} \sim \underline{v}$. \end{remark}


Let $X$ be a set; we denote $\mathcal{F}(X)$ the free group on $X$ and we consider $\mathcal{F}(X)$ as a subset\footnote{Usually $\mathcal{F}(X)$ is considered a quotient of $\mathcal{M}(X \cup X^{-1})$, but in this paper we will not follow this habit.} of $\mathcal{M}(X \cup X^{-1})$. In particular $\mathcal{F}(X)$ is the set of \textit{reduced words on $X$}, i.e., the words of the form $x_1 \dots x_n$ with $x_i \in X \cup X^{-1}$ and $x_{i+1} \neq x_i^{-1}$ for $i = 1, \dots, n-1$.

Given $w := y_1 \dots y_n \in \mathcal{M}(X \cup X^{-1})$, with $y_i \in X \cup X^{-1}$ for $i = 1, \dots, n$, the inverse of $w$ is the word $w^{-1} := y_n^{-1} \dots y_1^{-1}$. Of course $w^{-1} \in \mathcal{F}(X)$ if and only if $w \in \mathcal{F}(X)$.

\begin{definition} \label{} \rm We denote $\rho: \mathcal{M}(X \cup X^{-1}) \rightarrow \mathcal{F}(X)$ the function sending a word to its unique \textit{reduced form}. Given $u, v \in \mathcal{F}(X)$ the product\footnote{The product of two reduced words in $\mathcal{F}(X)$ does not coincide with the product of the same words in $\mathcal{M}(X \cup X^{-1})$. In particular $\mathcal{F}(X)$ is not a subgroup of $\mathcal{M}(X \cup X^{-1})$.} of $u$ by $v$ in $\mathcal{F}(X)$ is defined as $u \cdot v := \rho(uv)$, i.e., it is equal to the reduced form of the concatenation $uv$. \end{definition}


\begin{convention} \label{conv=} \rm In this paper we adopt the following conventions:

1. With the term \textit{word} we mean a non-necessarily reduced word, i.e., an element of $\mathcal{M}(X \cup X^{-1})$.

2. Given $u$, $v_1, \dotsc, v_n \in \mathcal{M}(X \cup X^{-1})$, with the notation $u = v_1 \dots v_n$ we mean the equality in $\mathcal{M}(X \cup X^{-1})$ of $u$ with the concatenation of words $v_1 \dots v_n$ even if $u$ and all the $v_j$ belong to $\mathcal{F}(X)$. This kind of equality is called a \textit{factorization} of $v$ in the Combinatorics of Words literature (see \cite{Kar}, pag. 2 or \cite{ChofKar}, pag. 332). The equality in $\mathcal{F}(X)$ of $u$ with the reduced product of $v_1, \ldots, v_n$ will be denoted by $u = v_1 \cdot \, \ldots \, \cdot v_n$ and corresponds to the equality $\rho(u) = \rho(v_1 \ldots v_n)$ in $\mathcal{M}(X \cup X^{-1})$. 
\end{convention}

\begin{remark} \label{reducprod} \rm Given $u, v \in \mathcal{F}(X)$ there exist $u_1, v_1, a \in \mathcal{F}(X)$ ($a$ can be equal to 1) such that $u=u_1 a$, $v=a^{-1} v_1$ and $\rho(uv)=u_1 v_1$. Therefore $uv=u_1 a a^{-1} v_1$.   \end{remark}

\begin{remark} \label{obvi1} \rm It is obvious that a word is reduced if and only if its inverse is reduced if and only if its reverse is reduced. \end{remark}

\begin{remark} \label{obviRho} \rm The reduced form of a word is unique, no matter the sequence of cancellations we perform on that word to obtain its reduced form (see Theorem 1.2 in Chapter 1 of \cite{MKS}). This implies that for words $u, v$ we have that $\rho(uv) = \rho(\rho(u) \rho(v))$. 
 \end{remark}

\begin{remark} \label{necToCpc} \rm The following fact is obvious. Let $w$ be a word and let $v_1, v_2$ be words such that $\rho(w) = v_1 v_2$. Then there exist words $w_1, w_2$ such that $w = w_1 w_2$ and $\rho(w_1) = v_1$, $\rho(w_2) = v_2$. \end{remark}

\begin{remark} \label{cpCanc} \rm Let $w$ be a word and let $u$ be a cyclic permutation of $\rho(w)$. Then there exists a cyclic permutation $w'$ of $w$ such that $\rho(w') = \rho(u)$.
	
Indeed there exist words $v_1, v_2$ such that $\rho(w) = v_1 v_2$ and $u = v_2 v_1$. By Remark \ref{necToCpc} there exist words $w_1, w_2$ such that $w = w_1 w_2$ and $\rho(w_1) = v_1$, $\rho(w_2) = v_2$. Let us set $w' := w_2 w_1$.

We have that $\rho(w') = \rho(w_2 w_1) = \rho(\rho(w_2) \rho(w_1)) = \rho(v_2 v_1) = \rho(u)$, where we have used Remark \ref{obviRho}.
\end{remark}

\begin{proposition} \label{rfrev} Let $w$ be a word. Then $\rho(\underline{w}) = \underline{\rho(w)}$ and the cancellations made to obtain $\rho(\underline{w})$ from $\underline{w}$ are the reverse of those made when obtaining $\rho(w)$ from $w$.\end{proposition}
	
\begin{proof} Let $n$ be the number of cancellations needed to obtain $\rho(w)$ from $w$; we prove the claim by induction on $n$.

First let $n = 1$. Then $w = w_1 x x^{-1} w_2$ for words $w_1, w_2, x$ and $\rho(w) = w_1 w_2$. We have that 
	$$\rho(\underline{w}) = \rho(\underline{w_2} \, \underline{x}^{-1} \, \underline{x} \, \underline{w_1}) = \rho(\underline{w_2} \, \underline{w_1}) = \rho(\underline{w_1 w_2}) = \underline{w_1 w_2} = \underline{\rho(w)},$$
where we have used Remark \ref{reverse} and the fact that since $w_1 w_2$ is reduced then by Remark \ref{obvi1} the word $\underline{w_1 w_2}$ is reduced. This proves the first part of the claim. For the second part we observe that we have the cancellation $x x^{-1}$ in $\rho(w)$ and its reverse $\underline{x}^{-1} \underline{x}$ in $\rho(\underline{w})$.

Now let $n > 1$ and the claim be true if the number of cancellations is less than $n$. In this case there exist words $w_1, w_2, x$ such that $w = w_1 x x^{-1} w_2$, $\rho(w) = \rho(w_1) \rho(w_2)$ (in particular the cancellation $x x^{-1}$ is the last one when obtaining $\rho(w)$ from $w$). By induction hypothesis since the cancellations in $w' := w_1 w_2$ are less than $n$, then we have that $\rho(\underline{w_1 w_2}) = \underline{\rho(w_1 w_2)}$, therefore
	$$\rho(\underline{w}) = \rho(\underline{w_2} \, \underline{x}^{-1} \, \underline{x} \, \underline{w_1}) = \rho(\underline{w_2} \, \underline{w_1}) =$$
	$$\rho(\underline{w_1 w_2}) = \underline{\rho(w_1 w_2)} = \underline{\rho(w_1 x x^{-1} w_2)} = \underline{\rho(w)},$$
proving the first part of the claim. For the second part we observe that the cancellations in $\rho(\underline{w})$ are those in $\rho(\underline{w_2})$, plus those in $\rho(\underline{w_1})$ and $\underline{x}^{-1} \underline{x}$. These cancellations are the reverse of those in $w$, which are the cancellations in $w_1$, those in $w_2$ and $x x^{-1}$. This proves the second part of the claim.	
\end{proof}


\begin{remark} \label{permCon} \rm Let $u$ and $v$ be words such that $u$ is a cyclic permutation of $v$; then the reduced form of $u$ is the reduced form of some conjugate of $v$.

Indeed there exist words $p, q$ such that $u = pq$ and $v = qp$; then $\rho(u) = \rho(pq) = \rho(p q p p^{-1}) = \rho(p v p^{-1})$. \end{remark}

\begin{remark} \label{uvw} \rm The following result is obvious: if $u, v, w$ are words such that $uv$ and $vw$ are reduced and $v \neq 1$ then $uvw$ is reduced. \end{remark}

\section{Cyclically reduced words and the cyclically reduced product} \label{S2}

In this section we introduce and study the basic properties of cyclically reduced words and the cyclically reduced product.

\begin{definition} \label{cyrewo}  \rm A reduced word is \textit{cyclically reduced} if its last letter is not the inverse of the first one, that is if all its cyclic permutations are reduced. We denote $\hat{\mathcal{F}}(X)$ the set of cyclically reduced words.   \end{definition}

\begin{remark} \label{obvi2} \rm It is obvious that a word is cyclically reduced if and only if its inverse is cyclically reduced if and only if its reverse is cyclically reduced. \end{remark}

\begin{remark} \label{obvi3} \rm Let $u$ be a cyclically reduced word; then the concatenation $u^2 = u u$ is also cyclically reduced. In general for any natural number $n$ the concatenation $u^n$ is cyclically reduced. \end{remark}

\begin{remark} \label{scope} \rm Given a word $w$ there exist unique $t \in \mathcal{F}(X)$ and $c \in \hat{\mathcal{F}}(X)$ such that $\rho(w) = t c t^{-1}$. The word $c$ is called the \textit{cyclically reduced form of $w$} and is denoted $\hat{\rho}(w)$. In particular we have the function 
	$$\hat{\rho}: \mathcal{M}(X \cup X^{-1}) \rightarrow \hat{\mathcal{F}}(X)$$ 
sending a word to its unique cyclically reduced form.
	
We have that $\rho(w) = t \hat{\rho}(w) t^{-1}$ and that $\rho(w)$ is cyclically reduced if and only if $t=1$. 
	
Moreover $\hat{\rho}(w) = \rho(t^{-1} w t)$, that is the cyclically reduced form of a word is equal to the reduced form of some conjugate of that word.
\end{remark}

\noindent \textbf{Example.} The word $x y z x y^{-1} x^{-1}$ is reduced but not cyclically reduced. Its cyclically reduced form is equal to $z x$. 

\medskip

\noindent \textbf{Remark.} In order to obtain the cyclically reduced form $\hat{\rho}(w)$ from a word $w$, first we have to make all ``internal" cancellations and after the ``external" cancellations. If we do not respect this order and we do some external cancellation before having done all the internal ones then we obtain a word which is a cyclic permutation of $\hat{\rho}(w)$. Theorem \ref{ordCanc} proves this fact.

\begin{remark} \label{yaObv} \rm It follows trivially from the definition that $\hat{\rho}(w) = \hat{\rho}(\rho(w))$ for a word $w$. \end{remark} 

\begin{remark} \label{} \rm We can say that in $\hat{\mathcal{F}}(X)$ a word is represented as a closed path, while in $\mathcal{F}(X)$ it is represented as an open path. This justifies why in $\mathcal{F}(X)$ we use the reduced form and in $\hat{\mathcal{F}}(X)$ the cyclically reduced form.

For example let us take the word $t x y y^{-1} z t^{-1}$. By representing it in a line as in Figure \ref{fig:wordLin}, the only consecutive opposite letters are $y y^{-1}$ and by canceling them we obtain the reduced form $t x z t^{-1}$.

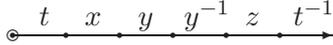
\begin{figure}[h!]

	\begin{picture}(80, 60)(20, 0)	 
	 	\put(140, 40){\vector(1,0){120}} 
	 
		\put(140, 40){\circle{4}}
		
		\put(140, 40){\circle*{2}}	
		\put(160, 40){\circle*{2}}	
		\put(180, 40){\circle*{2}}	
		\put(200, 40){\circle*{2}}	
		\put(220, 40){\circle*{2}}	
		\put(240, 40){\circle*{2}}
		
		\put(150, 44){$t$}	
		\put(167, 44){$x$}	
		\put(187, 44){$y$}	
		\put(204, 44){$y^{-1}$}	
		\put(227, 44){$z$}	
		\put(245, 44){$t^{-1}$}	
	\end{picture}

	\vspace*{-8mm}
	\caption{a word represented linearly}
	\label{fig:wordLin}
\end{figure}

On the other hand if we represent the same word in a cycle as in Figure \ref{fig:wordCyc}, not only the edges labeled by $y$ and $y^{-1}$ are contiguous, but also the edges labeled by $t^{-1}$ and $t$. The cyclically reduced form is $x z$ and its representation as a cycle is shown in Figure \ref{fig:wordCycCanc}. We observe that in order to obtain the reduced form from the cyclically reduced one we have to join a spine labeled $t t^{-1}$ to the initial point, as in Figure \ref{fig:wordCycCancSpine}.

\begin{figure}[h!]
	\begin{minipage}[b]{0.4\textwidth}
		\begin{picture}(80, 80) (100, 0) 
			\put(177, 40){\circle{38}}
			
			\put(157, 40){\circle*{2}}
			\put(167, 57){\circle*{2}}
			\put(187, 57){\circle*{2}}
			\put(197, 40){\circle*{2}}
			\put(187, 23){\circle*{2}}
			\put(167, 23){\circle*{2}}
			
			\put(175, 56){\vector(1,0){4}}
			\put(175, 42){\oval(27,27)[tl]}
			
			\put(157, 40){\circle{4}}
			
			\put(152, 49){$t$}
			\put(174, 63){$x$}
			\put(196, 51){$y$}
			\put(196, 26){$y^{-1}$}
			\put(174, 11){$z$}
			\put(144, 23){$t^{-1}$}
		\end{picture}
	\caption{a word represented as a cycle}
	\label{fig:wordCyc}
	\end{minipage}
	\hfill
	\begin{minipage}[b]{0.4\textwidth}
		\begin{picture}(80, 70) (110, 0) 
			\put(177, 40){\circle{38}}
			\put(157, 40){\line(1,0){40}}
		
			\put(157, 40){\circle{4}}
			\put(157, 40){\circle*{2}}
			\put(177, 40){\circle*{2}}
			\put(197, 40){\circle*{2}}
		
			\put(157, 57){$x$}
			\put(177, 60){\vector(1,0){4}}
		
			\put(184, 40){\vector(-1,0){4}}
		 	\put(187, 44){$y$}
		
			\put(166, 40){\vector(-1,0){4}}
		 	\put(169, 42){$t$}
		
			\put(177, 20){\vector(-1,0){4}}
			\put(195, 23){$z$}
		\end{picture}
		\vspace*{1mm}
		\caption{the cyclically reduced form of the same word}
		\label{fig:wordCycCanc}
	\end{minipage}
	\hfill
	\centering
	\begin{minipage}[b]{0.4\textwidth}
				\vspace*{7mm}
		\begin{picture}(80, 80) (95, 0) 
		
			\put(177, 40){\circle{38}}
			\put(127, 40){\line(1,0){70}}
		
			\put(127, 40){\circle{4}}
			\put(127, 40){\circle*{2}}
			\put(157, 40){\circle*{2}}
			\put(177, 40){\circle*{2}}
			\put(197, 40){\circle*{2}}
		
			\put(145, 40){\vector(1,0){4}}
		 	\put(134, 42){$t$}
		
			\put(157, 57){$x$}
			\put(177, 60){\vector(1,0){4}}
		
			\put(184, 40){\vector(-1,0){4}}
		 	\put(187, 44){$y$}
		
			\put(166, 40){\vector(-1,0){4}}
		 	\put(169, 42){$t$}
		
			\put(177, 20){\vector(-1,0){4}}
			\put(195, 23){$z$}
		\end{picture}
		\vspace*{-3mm}
		\caption{a spine is added to retrieve the reduced form}
		\label{fig:wordCycCancSpine}
	\end{minipage}
\end{figure}

\label{grafRep}

Given a representation of a word, it is easy to find the representation of its reverse or of a cyclic permutation. 

Let $w$ be represented linearly as in Figure \ref{fig:wordLin}; then its reverse has the same representation except that the direction is the opposite and the initial and final points are swapped. 

Let $w$ be represented as a cycle like in Figure \ref{fig:wordCyc}; then the reverse has the same representation as $w$ except the direction which is the opposite. The representation of a cyclic permutation of $w$ is the same as that of $w$ but the initial point is different.

From this we can also see the relationship between the representation of a word as a cycle and the concept of \textit{cyclic words}, also called \textit{circular words}, which are classes of equivalence of words under cyclic permutation. Indeed a cyclic word can be represented as in Figure \ref{fig:wordCyc} but without specifying the initial point.

\smallskip

We observe that in Figure \ref{fig:wordCycCanc} we have folded the edges corresponding to a cancellation and these edges are not part of the boundary label of the diagram. This is the same as in van Kampen diagrams (see Remark \ref{vanKamp} for a short and informal presentation of van Kampen diagrams or one of the many references in the literature like \cite{Short} for a more thorough one). 

The last sentence needs however a clarification. Indeed we can consider that the diagrams in Figures \ref{fig:wordCyc}-\ref{fig:wordCycCancSpine} are van Kampen diagrams where we have omitted to draw all internal edges, so not all regions are specified: indeed in a van Kampen diagram the labels of the regions must be basic relators and also the number of regions does not increase when we make a folding corresponding to a cancellation\footnote{indeed in Figure \ref{fig:wordCyc} there is one ``region'', while there are two in Figure \ref{fig:wordCycCanc}}. 

We can go even farther by saying that the diagram in Figure \ref{fig:wordCyc} represents the boundary of a van Kampen diagram of any type (with also spines) and that the process of folding edges done in Figure \ref{fig:wordCycCanc} for a word represented by any van Kampen diagram is compatible with that of folding edges in van Kampen diagrams (see the reference cited above for more details about van Kampen diagrams).
 \end{remark}

\begin{remark} \label{permCon2} \rm Let $w$ be a word and let $u$ be a cyclic permutation of $\hat{\rho}(w)$; then $u$ is the reduced form of some conjugate of $w$. 
	
Indeed there exist words $p, q$ such that $\hat{\rho}(w) = pq$ and $u = qp$. Then $u = \rho(p^{-1} p q p) = \rho(p^{-1} \hat{\rho}(w) p)$.

By Remark \ref{scope} we have that $\hat{\rho}(w) = \rho(t^{-1} w t)$ for some word $t$, so $u = \rho(p^{-1} \hat{\rho}(w) p) = \rho(p^{-1} t^{-1} w t p)$.

This implies one side of the well known equivalence which gives a solution to the conjugacy problem in free groups, (see Theorem 1.3 in Chapter 1 of \cite{MKS} or Proposition I.2.14 of \cite{LS}): two words are cyclic conjugate if and only if their cyclically reduced forms are cyclic permutations one of the other. Corollary \ref{permCycRedForm} implies the other side of this equivalence. 
\end{remark}

\begin{remark} \label{barrho=1} \rm Let $w$ be a word; then:
	\begin{enumerate} [(1)]
		\item $\big(\hat{\rho}(w)\big)^{-1} = \hat{\rho}(w^{-1})$.
		
		\item $\hat{\rho}(w) = 1$ if and only if $\rho(w) = 1$.
	\end{enumerate}
Indeed (1) is obvious from the definition. To prove (2) we observe that by Remark \ref{scope} we have $\rho(w) = t \hat{\rho}(w) t^{-1}$ for some word $t$. If $\hat{\rho}(w) = 1$ then $\rho(w) = t t^{-1}$ and $t$ should be equal to 1 since $\rho(w)$ is reduced. The opposite implication is obvious.
\end{remark}

\begin{remark} \label{redCycRed} \rm If $u$ and $v$ are words such that $\rho(u) = \rho(v)$ then $\hat{\rho}(u) = \hat{\rho}(v)$, i.e., two words with equal reduced form have also equal cyclically reduced form. The inverse implication is false, for instance if $u$ is a reduced but not cyclically reduced word and if $v = \hat{\rho}(u)$ then  $\hat{\rho}(u) = \hat{\rho}(v)$ but $\rho(u) \neq \rho(v)$.  \end{remark}

\begin{remark} \label{revInv} \rm The operations of inversion and reversion of words commute one with the other, that is given a word $v$ we have that $(\underline{v})^{-1} = \underline{(v^{-1})}$. Thus we will denote $\underline{v}^{-1}$ the inverse of the reversion of $v$ without fear of ambiguity. If $v = x_1 \dots x_n$ then obviously $\underline{v}^{-1} = x_1^{-1} \dots x_n^{-1}$. 
	
Moreover given words $u$ and $u'$ we have that $u$ is a cyclic permutation of $u'$ if and only if $\underline{u}$	is a cyclic permutation of $\underline{u'}$. Therefore $v$ is (cyclically) reduced if and only if $\underline{v}$ is (cyclically) reduced if and only if $v^{-1}$ is (cyclically) reduced. \end{remark}

\begin{proposition} \label{revCycRedForm} Let $w$ be a word. Then $\hat{\rho}(\underline{w}) = \underline{\hat{\rho}(w)}$ and the cancellations made to obtain $\hat{\rho}(\underline{w})$ from $\underline{w}$ are the reverse of those made to obtain $\hat{\rho}(w)$ from $w$.\end{proposition}

\begin{proof} First let us assume that $w$ is reduced. By Remark \ref{scope} there exists a word $t$ such that $w = t \hat{\rho}(w) t^{-1}$, thus $\underline{w} = \underline{t}^{-1} \,\underline{\hat{\rho}(w)} \, \underline{t}$ and since $\underline{\hat{\rho}(w)}$ is cyclically reduced by Remark \ref{obvi2}, then $\hat{\rho}(\underline{w}) = \underline{\hat{\rho}(w)}$.

Now let $w$ be non-necessarily reduced. Then by Remark \ref{yaObv} and Proposition \ref{rfrev} we have
	$$\hat{\rho}(\underline{w}) = \hat{\rho}(\rho(\underline{w})) = \hat{\rho}(\underline{\rho(w)}) = \underline{\hat{\rho}(\rho(w))} = \underline{\hat{\rho}(w)}.$$
The cancellations made to obtain $\hat{\rho}(\underline{w})$ from $\underline{w}$ are those made to obtain $\rho(\underline{w})$ from $\underline{w}$ plus $\underline{t} \, \underline{t}^{-1}$. On the other hand the cancellations made to obtain $\hat{\rho}(w)$ from $w$ are those made to obtain $\rho(w)$ from $w$ plus $t^{-1} t$. The second part of the claim follows then from the fact that by Proposition \ref{rfrev} the cancellations made to obtain $\hat{\rho}(\underline{w})$ from $\underline{w}$ are the reverse of those made to obtain $\rho(w)$ from $w$.	
\end{proof}

\begin{definition} \label{key} \rm Given words $u$ and $v$ we denote $u*v$ the \textit{cyclically reduced product of $u$ by $v$}, i.e., $u * v := \hat{\rho}(uv)$. 
\end{definition}

\begin{remark} \label{jarem} \rm For every word $w$ we have that $w * 1 = 1 * w = \hat{\rho}(w)$, thus the word 1 is an identity in $\hat{\mathcal{F}}(X)$.

We show now that 1 is the unique idempotent (and thus the unique identity element) in $\hat{\mathcal{F}}(X)$. Indeed let $w \in \hat{\mathcal{F}}(X)$ be such that $w \neq 1$. Then since $w$ is cyclically reduced the word $ww$ is cyclically reduced and thus $w*w = ww \neq w$.

We have also that for each element $u \in \hat{\mathcal{F}}(X)$ the inverse of $u$ in $\mathcal{F}(X)$ is the unique left inverse and the unique right inverse of $u$ in $\hat{\mathcal{F}}(X)$. Indeed (2) of Remark \ref{barrho=1} implies that $u*v = 1$ if and only if $\rho(u v) = 1$, thus if and only if $\rho(v) = \rho(u^{-1})$. If $u$ is reduced (in particular if it is cyclically reduced) this implies that $u^{-1}$ is the unique right inverse of $u$. The same argument shows that $u^{-1}$ is the unique left inverse of $u$.
\end{remark}

\begin{remark} \label{u*vNonRed} \rm We have that $u*v = \rho(u)*\rho(v)$. More generally let $u_1, v_1$ be words such that $\rho(u_1) = \rho(u)$ and $\rho(v_1) = \rho(v)$. Then $u*v = u_1 * v_1$.
	
Indeed by Remarks \ref{redCycRed} and \ref{obviRho} we have that
	$$u*v = \hat{\rho}(uv) = \hat{\rho}\big(\rho(uv)\big) = \hat{\rho}\Big(\rho\big(\rho(u) \rho(v)\big)\Big) =$$
	$$\hat{\rho}\Big(\rho\big(\rho(u_1) \rho(v_1)\big)\Big) = \hat{\rho}\big(\rho(u_1 v_1)\big) = \hat{\rho}(u_1 v_1) = u_1 * v_1.$$
\end{remark}

\begin{remark} \label{revCRP} \rm If $u$ and $v$ are words then the reverse of $u*v$ is equal to $\underline{v}*\underline{u}$ and the cancellations made to obtain $\underline{v}*\underline{u}$ from $\underline{v} \, \underline{u}$ are the reverse of those made to obtain $u*v$ from $uv$.
	
Indeed by Remark \ref{reverse} and Proposition \ref{revCycRedForm}
	$$\underline{v} * \underline{u} = \hat{\rho}(\underline{v} \, \underline{u}) = \hat{\rho}(\underline{uv}) = \underline{\hat{\rho}(uv)} = \underline{u*v},$$
proving the first part of the claim. The second part follows from Proposition \ref{revCycRedForm}.
\end{remark}

\section{Properties of $\hat{\mathcal{F}}(X)$} \label{propHatF}

In this section we discuss the structure of $\hat{\mathcal{F}}(X)$ equipped with the cyclically reduced product.

\smallskip

$\hat{\mathcal{F}}(X)$ is closed with respect to $*$, so $(\hat{\mathcal{F}}(X), *)$ is a \textit{magma}\footnote{A magma is a set closed with respect to an operation, see \cite{EncyMath}.}. The word 1 is the unique identity element and is also the unique idempotent; moreover for every $w \in \hat{\mathcal{F}}(X)$ the element $w^{-1}$ is the unique inverse of $w$.
In particular $w^{-1}$ is the unique left and the unique right inverse of $w$ (see Remark \ref{jarem}).

If $X$ has only one element, then a word is cyclically reduced if and only if it is reduced, so in that case $\hat{\mathcal{F}}(X) = \mathcal{F}(X)$ and the cyclically reduced product is the same as the reduced product.

Let $X$ have at least two elements; then $\hat{\mathcal{F}}(X)$ is a proper subset of $\mathcal{F}(X)$ and moreover $*$ is not associative, therefore $(\hat{\mathcal{F}}(X), *)$ is not a group. Indeed let $x, y \in X$ be such that $x \neq y$ and let us set $u := xy$, $v := x^{-1}$ and $w := x$; then $(u*v)*w = yx$ while $u*(v*w) = xy$. 

However a more general result holds true: in \cite{TwAs1} and \cite{TwAs2} we prove that $*$ verifies a ``twisted'' (i.e., up to cyclic permutations) version of the associative property.

The operation * is also not commutative but we prove in Theorem \ref{puzo} that for any $u, v$ we have that $u*v$ is a cyclic permutation of $v*u$ (thus a ``twisted'' commutative property holds true, using the terminology introduced above).


\bigskip

\noindent \textbf{Latin square and cancellation properties.} 

\smallskip

Since $\mathcal{F}(X)$ is a group we know that given $u, w \in \mathcal{F}(X)$ there exist unique $v_1, v_2 \in \mathcal{F}(X)$ such that $u \cdot v_1 = w$ and $v_2 \cdot u = w$, in particular $v_1 = u^{-1} \cdot w$ and $v_2 = w \cdot u^{-1}$. Moreover $v_1$ and $v_2$ are conjugates.

This property is a generalization of the existence of an inverse and is called \textit{Latin square property}, see \cite{Jordan}. 

It is natural to ask whether that property holds true also for $\hat{\mathcal{F}}(X)$ equipped with *. The answer is no if $w \neq 1$, so $\hat{\mathcal{F}}(X)$ is not a \textit{quasi-group}\footnote{A quasi-group is a magma where the Latin square property holds, see \cite{EncyMath}).}. Indeed as Corollary \ref{ax=b2} shows, $\hat{\mathcal{F}}(X)$ verifies a weaker property than the Latin square: if $u, w \neq 1$ there exist infinitely many pairs of words $v_1$ and $v_2$ such that $v_1$ and $v_2$ are cyclic permutations one of the other and $u * v_1 = v_2 * u = w$. In particular it is the uniqueness that does not hold in $\hat{\mathcal{F}}(X)$.

This result also implies that no element of $\hat{\mathcal{F}}(X)$ has either the left or the right cancellation property. It also implies that even if 1 is the unique identity element of $\hat{\mathcal{F}}(X)$, then for any $u \in \hat{\mathcal{F}}(X)$ such that $u \neq 1$ there exist infinitely many $v_1$ and $v_2$ such that $u * v_1 = u$ and $v_2 * u = u$.

\smallskip

We observe that if we set $v_1 := u^{-1} * w$ and $v_2 := w * u^{-1}$, it is not true in general that $u * v_1 = w$ and $v_2 * u = w$. Indeed let $x, y \in X$ be such that $x \neq y$ and let $u := xy$ and $w := y^2$; then $u * v_1 = x y x^{-1} y \neq w$, while $v_2 * u = w$. Let $u := xy$ and $w := x^2$; then $u * v_1 = w$, while $v_2 * u = x y^{-1} x y \neq w$. Finally if $u := y x y$ and $w := y x^{-1} y$, then $u * v_1 = y x y x^{-2} \neq w$, $v_2 * u = x^{-2} y x y \neq w$ and $u * v_1 \neq v_2 * u$.

However a more general result holds true: in \cite{TwAs1} we prove that there exists a cyclic permutation $u'$ of $u$ such that either $u'*v_1$ is a cyclic permutation of $w$ or there exists a word $h$ such that $u'hv_1h^{-1}$ is cyclically reduced and is a cyclic permutation of $w$. A symmetrical result holds true for $v_2$ in place of $v_1$.

\section{A generalization of a result holding in the free group} \label{S3}

This section uses results and definitions of Appendices \ref{AppA} and \ref{AppB}. In particular in Section \ref{AppA} identities among relations are defined.

We recall that $\rho$ and $\hat{\rho}$ denote respectively the reduced form and the cyclically reduced form of a word and that $*$ denotes the cyclically reduced product of two words.

The next theorem generalizes to the cyclically reduced product the following result: given words $u, v$ the reduced product $u \cdot v =\rho(uv)$ is a cyclic conjugate of $v \cdot u$ and the identity among relations that by Remark \ref{idFrom} follows from this fact semi-Peiffer collapses to 1. 

Indeed we have that $\rho(uv) = \rho\big(u(vu)u^{-1}\big)$, which implies the identity among relations
	$$u \centerdot v \equiv uvu^{-1} \centerdot u u u^{-1},$$
which is equivalent to the identity in normal form
	$$u \centerdot v \centerdot u u^{-1} u^{-1} \centerdot u v^{-1} u^{-1} \equiv 1.$$
The element of $H$ corresponding\footnote{$H$ is defined at page \pageref{H} in appendix \ref{AppA}} to this equivalence is 
	\begin{equation} \label{eqBefThCom} (1, u) \, (1, v) \, (u, u^{-1}) \, (u, v^{-1}),\end{equation} 
which by means of an exchange of type A-2 transforms to 
	$$(1, u) \, (u, u^{-1}) \, (u, v) \, (u, v^{-1}).$$
The latter reduces to 1 with one semi-Peiffer deletion between the first and second term and one Peiffer deletion between the third and fourth.

We will prove that for the cyclically reduced product we have a generalization of the same result with cyclic permutation instead of cyclic conjugation. In particular the identity among relations is a generalization of that of (\ref{eqBefThCom}). Moreover we have a result that does not hold for the reduced product, namely that the cancellations in $u*v$ are the same up to cyclic permutation as those in $v*u$. 

As we will se in Remark \ref{vanKamp}, the latter fact implies that if $u$ and $v$ are relators of a group presentation, then the van Kampen diagrams associated with $u*v$ and $v*u$ are homeomorphic possibly with different initial points. This of course does not hold for the reduced product because if $u \cdot v$ and $v \cdot u$ are not cyclic permutation one of the other, even their boundaries cannot be homeomorphic.

\begin{theorem} \label{puzo} Let $u$ and $v$ be words such that $u*v \neq 1$; then
\begin{enumerate} [i)]
	\item $u*v$ is a cyclic permutation of $v*u$.
\end{enumerate}
 Now let $u$ and $v$ be reduced; then we have also that 
\begin{enumerate} [i)]
  \setcounter{enumi}{1}
	\item the words canceled when obtaining $u*v$ from $uv$ are the same up to cyclic permutation as those canceled when obtaining $v*u$ from $vu$;

	\item the identity among relations involving $u, v, u^{-1}, v^{-1}$ that by Remark \ref{idFrom} follows from the equivalence in \textit{i}) is cyclic permutation either in the first and third terms or in the second and fourth terms;

	\item that identity semi-Peiffer collapses to 1 by means of the following sequence of $2n+3$ operations: $n$ exchanges of type B-1; $n$ exchanges of type B-3; an exchange of type A-2; a semi-Peiffer deletion between the first and second term; a Peiffer deletion between the third and fourth terms;

	\item if there exist words $\alpha$, $\beta$, $u'$, $v'$ such that $u = \alpha u' \beta$ and $v = \beta^{-1} v' \alpha^{-1}$ then the words $\beta \beta^{-1}$ and $\alpha^{-1} \alpha$ are canceled when obtaining $u*v$ from $uv$ and when obtaining $v*u$ from $vu$.
\end{enumerate}
\end{theorem}

\begin{proof} First let us assume that $u$ and $v$ are reduced: we show that the claims are true for the three cases of Lemma \ref{shirv}. 
	
\textbf{1.} We have that $u = u_1 a$, $v = a^{-1} s (u*v) s^{-1} u_1^{-1}$, 
	$$\rho(uv) = u_1 s (u*v) s^{-1} u_1^{-1},$$ 
which implies that 	
	\begin{equation} \label{eqThCom1.1} u*v = \rho(s^{-1} u_1^{-1} uv u_1 s)\end{equation}
Moreover since $u*v$ is cyclically reduced then in obtaining $u*v$ from $uv$ we have the internal cancellation $a a^{-1}$ and the external cancellations $u_1^{-1} u_1$ and $s^{-1} s$.

Also we have that
	$$\rho(vu) = \rho(a^{-1} s (u*v) s^{-1} a)$$			
and there is the internal cancellation $u_1^{-1} u_1$ when obtaining $\rho(vu)$ from $vu$.

We have that $a^{-1} s (u*v)$ is a subword of $v$, thus it is reduced. By (2) of Lemma \ref{complic} there exist words $w_1, w_2, b_1$ and a natural number $n$ such that $u*v = w_1 w_2$, $\rho(a^{-1} s (u*v) s^{-1} a) = b_1 w_2 w_1 b_1^{-1}$, $w_1 \neq 1$ and $a^{-1} s = b_1 w_2 (u*v)^n$.

This implies that 
	$$s^{-1} a = (u*v)^{-n} w_2^{-1} b_1^{-1}, \hspace{0.3cm} a = \rho(s (u*v)^{-n} w_2^{-1} b_1^{-1})$$
	$$u = \rho(u_1 s (u*v)^{-n} w_2^{-1} b_1^{-1}), \hspace{0.3cm} v = b_1 w_2 (u*v)^{n+1} s^{-1} u_1^{-1}$$



We have that 
	$$v*u = \hat{\rho}(vu) = \hat{\rho}(a^{-1} s (u*v) s^{-1} a) = \hat{\rho}(b_1 w_2 w_1 b_1^{-1}) = w_2 w_1,$$
so $v*u$ is a cyclic permutation of $u*v$, proving \textit{i)}, in particular 	
	\begin{equation} \label{temp1} u*v  = \rho(w_2^{-1} (v*u) w_2).\end{equation}
Also since $\rho(vu) = b_1 w_2 w_1 b_1^{-1}$, then	
	\begin{equation} \label{temp2} v*u = \rho(b_1^{-1} vu b_1).\end{equation}
(\ref{temp1}) and (\ref{temp2}) together with (\ref{eqThCom1.1}) give the following identity among the relations
	$$(s^{-1} u_1^{-1}) u (u_1 s) \centerdot (s^{-1} u_1^{-1}) v (u_1 s) \equiv (w_2^{-1} b_1^{-1}) v (b_1 w_2)  \centerdot (w_2^{-1} b_1^{-1}) u (b_1 w_2),$$
which is equivalent to
	\begin{equation} \label{temp3} (s^{-1} u_1^{-1}) u (u_1 s) \centerdot (s^{-1} u_1^{-1}) v (u_1 s) \centerdot (w_2^{-1} b_1^{-1}) u^{-1} (b_1 w_2) \centerdot (w_2^{-1} b_1^{-1}) v^{-1} (b_1 w_2) \equiv 1.\end{equation}
This identity is cyclic permutation in the second and fourth terms\footnote{If $s = 1$ it is cyclic permutation in all terms.}, proving \textit{iii)}.

\smallskip

Now let us prove \textit{ii)}. 

Let $s \neq 1$; then $(u*v) s^{-1} a$ is reduced by Remark \ref{uvw} because: $(u*v) s^{-1}$ is a subword of $v$, so it is reduced; $s^{-1} a$ is a subword of $v^{-1}$, so it is reduced; $s^{-1} \neq 1$. This implies that $a^{-1} s (u*v) s^{-1} a$ is reduced by Remark \ref{uvw} because: $a^{-1} s (u*v)$ is a subword of $v$, so it is reduced; $(u*v) s^{-1} a$ is reduced by what seen above; $(u*v) \neq 1$ by hypothesis.

Thus if $s \neq 1$ then $\rho(vu) = a^{-1} s (u*v) s^{-1} a$, therefore $v*u = u*v$ and the cancellations made to obtain $v*u$ from $vu$ are, besides the internal cancellation $u_1^{-1} u_1$, also the external cancellations $a a^{-1}$ and $s^{-1} s$, proving \textit{ii)} for $s \neq 1$.

Now let $s = 1$; then since $u*v = w_1 w_2$, then $a = (w_1 w_2)^{-n} w_2^{-1} b_1^{-1}$, $u = u_1 (w_1 w_2)^{-n} w_2^{-1} b_1^{-1}$, so 
	$$\rho(vu) = \rho(b_1 w_2 (w_1 w_2)^{n+1} u_1^{-1} u_1 (w_1 w_2)^{-n} w_2^{-1} b_1^{-1}) =$$
	$$\rho(b_1 w_2 (w_1 w_2) w_2^{-1} b_1^{-1}) = \rho(b_1 w_2 w_1 b_1^{-1})$$
and we have the internal cancellations $u_1^{-1} u_1$, $(w_1 w_2)^n (w_1 w_2)^{-n}$ and $w_2 w_2^{-1}$.

Since by what seen above $\rho(vu) = \rho(a^{-1} s (u*v) s^{-1} a)$ and $\rho(a^{-1} s (u*v) s^{-1} a) = b_1 w_2 w_1 b_1^{-1}$, then $\rho(vu) = b_1 w_2 w_1 b_1^{-1}$, therefore when obtaining $v*u$ from $\rho(vu)$ we have the external cancellation $b_1^{-1} b_1$.

In conclusion, since $a = (w_1 w_2)^{-n} w_2^{-1} b_1^{-1}$, when obtaining $u*v$ from $uv$ we have the internal cancellations $b_1^{-1} b_1$, $w_2^{-1} w_2$ and $(w_1 w_2)^{-n} (w_1 w_2)^n$ and the external cancellation $u_1^{-1} u_1$. When obtaining $v*u$ from $vu$ we have the internal cancellations $u_1^{-1} u_1$, $(w_1 w_2)^n (w_1 w_2)^{-n}$ and $w_2 w_2^{-1}$ and the external cancellation $b_1^{-1} b_1$, proving \textit{ii)} for $s = 1$.

\smallskip


Now let us prove \textit{iv)}. The element of $H$ (see section \ref{AppA}) associated with the left hand side of (\ref{temp3}) is 
	$$h := [(s^{-1} u_1^{-1}, u), (s^{-1} u_1^{-1}, v), (w_2^{-1} b_1^{-1}, u^{-1}), (w_2^{-1} b_1^{-1}, v^{-1})].$$
We have that: 
\begin{enumerate}
	\item [--] $\rho(s^{-1} u_1^{-1} u u_1 s) = \rho(s^{-1} u_1^{-1} u_1 s (u*v)^{-n} w_2^{-1} b_1^{-1} u_1 s) =$

	$\rho((u*v)^{-n} w_2^{-1} b_1^{-1} u_1 s)$;

	\item [--] $\rho(s^{-1} u_1^{-1} v u_1 s) = \rho(s^{-1} u_1^{-1} b_1 w_2 (u*v)^{n+1} s^{-1} u_1^{-1} u_1 s) =$

	$\rho(s^{-1} u_1^{-1} b_1 w_2 (u*v)^{n+1})$;

	\item [--] $\rho(w_2^{-1} b_1^{-1} u^{-1} b_1 w_2) = \rho(w_2^{-1} b_1^{-1} b_1 w_2 (u*v)^{n} u_1^{-1} s^{-1} b_1 w_2) =$

	$\rho((u*v)^{n} u_1^{-1} s^{-1} b_1 w_2)$;

	\item [--] $\rho(w_2^{-1} b_1^{-1} v^{-1} b_1 w_2) = \rho(w_2^{-1} b_1^{-1} u_1 s (u*v)^{-n-1} w_2^{-1} b_1^{-1} b_1 w_2) =$

	$\rho(w_2^{-1} b_1^{-1} u_1 s (u*v)^{-n-1})$.
\end{enumerate}

Let us set $p := u*v$ and $q := s^{-1} u_1^{-1} b_1 w_2$. Then we have that: 
\begin{enumerate}
	\item [--] $\rho(s^{-1} u_1^{-1} u u_1 s) = p^{-n} q^{-1}$;

	\item [--] $\rho(s^{-1} u_1^{-1} v u_1 s) = q p^{n+1}$;

	\item [--] $\rho(w_2^{-1} b_1^{-1} u^{-1} b_1 w_2) = p^{n} q$;

	\item [--] $\rho(w_2^{-1} b_1^{-1} v^{-1} b_1 w_2) = q^{-1} p^{-n-1}$.
\end{enumerate}

\textit{iv)} then follows from Lemma \ref{collapseh}.

Finally, to prove \textit{v)} we observe that $\alpha$ is a prefix of $u_1$, $\beta$ is a suffix of $a$, $\beta^{-1}$ is a prefix of $a^{-1}$ and $\alpha^{-1}$ a suffix of $u_1^{-1}$, so the claim is obvious.

\medskip

\textbf{2.} We have that $u = t c_1 a$, $v = a^{-1} c_2 t^{-1}$, $u*v = c_1 c_2$ and $\rho(uv) = t c_1 c_2  t^{-1}$, with $c_1, c_2 \neq 1$, thus $u*v = \rho(t^{-1} uv t)$ and when obtaining $u*v$ from $uv$ we have the internal cancellation $a a^{-1}$ and the external cancellation $t^{-1} t$.

On the other hand 
	$$\rho(vu) = \rho(a^{-1} c_2 t^{-1} t c_1 a) = a^{-1} c_2 c_1 a,$$
since $a^{-1} c_2 c_1 a$ is reduced by Remark \ref{uvw} because $a^{-1} c_2$, $c_2 c_1$ and $c_1 a$ are reduced and $c_1, c_2 \neq 1$. Moreover when obtaining $v*u$ from $vu$ we have an internal cancellation $t^{-1} t$ and an external cancellation $a a^{-1}$, proving \textit{ii)}. Also $v*u = c_2 c_1 = \rho(a vu a^{-1})$ and is a cyclic permutation of $u*v$, proving \textit{i)}.

Since $u*v = \rho(c_1 (v*u) c_1^{-1})$, we have the following identity among relations
	$$t^{-1} u t \centerdot t^{-1}v t \equiv (c_1 a) v (a^{-1} c_1^{-1}) \centerdot (c_1 a) u (a^{-1} c_1^{-1})$$
which is equivalent to 
	$$t^{-1} u t \centerdot t^{-1} v t \centerdot (c_1 a) u^{-1} (a^{-1} c_1^{-1}) \centerdot (c_1 a) v^{-1} (a^{-1} c_1^{-1})\equiv 1.$$
This identity is cyclic permutation in the first and third terms, proving \textit{iii)}.

We have that: 
\begin{enumerate}
	\item [--] $\rho(t^{-1} u t) = \rho(t^{-1} t c_1 a t) = \rho(c_1 a t)$;

	\item [--] $\rho(t^{-1} v t) = \rho(t^{-1} a^{-1} c_2 t^{-1} t) = \rho(t^{-1} a^{-1} c_2)$;

	\item [--] $\rho(c_1 a u^{-1} a^{-1} c_1^{-1}) = \rho(c_1 a a^{-1} c_1^{-1} t^{-1} a^{-1} c_1^{-1}) = \rho(t^{-1} a^{-1} c_1^{-1})$;

	\item [--] $\rho(c_1 a v^{-1} a^{-1} c_1^{-1}) = \rho(c_1 a t c_2^{-1} a a^{-1} c_1^{-1}) = \rho(c_1 a t c_2^{-1} c_1^{-1})$.
\end{enumerate} 

Let us set $p := c_1 c_2$ and $q := t^{-1} a^{-1} c_1^{-1}$. Then we have that: 
\begin{enumerate}
	\item [--] $\rho(t^{-1} u t) = \rho(p^{-n} q^{-1})$;

	\item [--] $\rho(t^{-1} v t) = \rho(q p^{n+1})$;

	\item [--] $\rho(c_1 a u^{-1} a^{-1} c_1^{-1}) = \rho(p^{n} q)$;

	\item [--] $\rho(c_1 a v^{-1} a^{-1} c_1^{-1}) = \rho(q^{-1} p^{-n-1})$.
\end{enumerate}

\textit{iv)} then follows from Lemma \ref{collapseh}.

Finally, to prove \textit{v)} we observe that $\alpha$ is a prefix of $t$, $\beta$ is a suffix of $a$, $\beta^{-1}$ is a prefix of $a^{-1}$ and $\alpha^{-1}$ a suffix of $t^{-1}$, so the claim is obvious.

\medskip

\textbf{3.} By swapping $u$ and $v$ this case reduces to case 1, since the claims are symmetrical in $u$ and $v$.

\medskip

Now let $u$ and $v$ be non-necessarily reduced. By the previous argument we have that $\rho(u)*\rho(v)$ is a cyclic permutation of $\rho(v)*\rho(u)$. Since $u*v = \rho(u)*\rho(v)$ and $v*u = \rho(v)*\rho(u)$ by Remark \ref{u*vNonRed} then 1. follows trivially.
\end{proof}

\begin{remark} \label{vanKamp} \rm Part \textit{ii)} of Theorem \ref{puzo} proves that the van Kampen diagrams associated with $u*v$ and $v*u$ are homeomorphic with non-necessarily the same initial point, so their boundary cycles are a cyclic permutation one of the other.

Let us introduce some notions to justify this claim. A van Kampen diagram (see \cite{Short}, Ch$\ldotp$ 2) is a planar 2-CW complex that can be associated with any relator of a group presentation. A relator is the reduced form of a product of conjugates of basic relators, and with the non-reduced product of these basic relators can be associated a van Kampen diagram in the form of a ``bouquet of lollipops" (see Fig$\ldotp$ 2.3 of \cite{Short}). Each letter of that product corresponds to an edge of the van Kampen diagram and every time there is a cancellation in that product, there is a folding of the two edges corresponding to the two canceled letters. When all the cancellations have been made, the obtained word, that we call $w$, is reduced and this word is the label of the boundary of the van Kampen diagram obtained after all the foldings. That boundary forms a cycle (i.e., its initial and final points coincide) and starts with the edge labeled by the first letter of $w$, so the initial point of that edge is the initial point of the van Kampen diagram.

Now let $u$ and $v$ be two relators. In order to find the van Kampen diagram associated with the reduced product $u \cdot v$, first we find the diagram associated with the non-reduced product $uv$: this is obtained by joining the van Kampen diagrams associated with $u$ and $v$ by making coincide their initial points. Then as above for each cancellation in the word $u v$ we fold together the two edges corresponding to the two canceled letters.

Now let us return to \textit{ii)} of Theorem \ref{puzo}. Let us take a group presentation of which $u$ and $v$ are relators (non-necessarily basic) and let us consider the van Kampen diagrams associated with them. Let us consider the van Kampen diagrams associated with $uv$ and $vu$: they are homeomorphic and their boundary cycles are cyclic permutations one of the other.

When we obtain $u*v$ from $uv$ we make first all internal cancellations and then all the external ones. When we do these cancellations we fold the corresponding edges of the diagram as explained above. Since the boundary of the van Kampen diagram is a cycle, an external cancellation corresponds to two edges that are consecutive, so the folding of these edges follows the same rules as for the internal cancellations. 

Since as shown in \textit{ii)} of Theorem \ref{puzo} the cancellations in $u*v$ are the same up to cyclic permutation as those of $v*u$ and since the van Kampen diagrams for $uv$ and $vu$ are homeomorphic with boundary cycles that are cyclic permutations one of the other, the same foldings applied to the diagram of $uv$ are applied to the diagram of $vu$.

This shows that the van Kampen diagrams for $u*v$ and $v*u$ are homeomorphic and their boundary cycles are cyclic permutation one of the other.
\end{remark}

\begin{corollary} \label{cycRedPerm} Let $w$ be a word and let $w'$ be a cyclic permutation of $w$. Then $\hat{\rho}(w')$ is a cyclic permutation of $\hat{\rho}(w)$.
\end{corollary}

\begin{proof} Since $w'$ is a cyclic permutation of $w$, there exist words $u, v$ such that $w = u v$ and $w' = v u$. Therefore $\hat{\rho}(w) = \hat{\rho}(uv) = u*v$ and $\hat{\rho}(w') = \hat{\rho}(vu) = v*u$ and by Theorem \ref{puzo} $\hat{\rho}(w')$ is a cyclic permutation of $\hat{\rho}(w)$.
\end{proof}

The next corollary implies one side of the equivalence that gives a solution to the conjugacy problem in free groups (Remark \ref{permCon2} implies the other side).

\begin{corollary} \label{permCycRedForm} If $t, w$ are words then $\hat{\rho}(t w t^{-1})$ is a cyclic permutation of $\hat{\rho}(w)$. If moreover $\rho(t)\rho(w)\rho(t)^{-1}$ is reduced then $\hat{\rho}(t w t^{-1}) =\hat{\rho}(w)$.
\end{corollary}

\begin{proof} By Corollary \ref{cycRedPerm} we have that $\hat{\rho}(t w t^{-1})$ is a cyclic permutation of $\hat{\rho}(w t t^{-1})$ and by Remark \ref{redCycRed} we have that $\hat{\rho}(w t t^{-1}) = \hat{\rho}(w)$, which proves the first part of the claim. 
	
Now let $\rho(t)\rho(w)\rho(t)^{-1}$ be reduced. Then $\hat{\rho}(t w t^{-1}) = \hat{\rho}(\rho(t)\rho(w)\rho(t)^{-1}) = \hat{\rho}(\rho(w)) = \hat{\rho}(w)$.
\end{proof}

We observe that Corollary \ref{permCycRedForm} cannot be improved, since in general $\hat{\rho}(t w t^{-1})$ is not equal to $\hat{\rho}(w)$, even when $t$ and $w$ are reduced words. Indeed let $x, y \in X$ and let $t := x$ and $w := yx$. Then $\hat{\rho}(t w t^{-1}) = \hat{\rho}(x y x x^{-1}) = xy \neq yx = \hat{\rho}(w)$.

\medskip

Let $w$ be a word and let $a, w_1, w_2$ be words such that $w = w_1 a a^{-1} w_2$; we say that the word $w' := w_1 w_2$ \textit{is obtained from $w$ by internally canceling $a a^{-1}$}. Let $t, w_0$ be words such that $w = t w_0 t^{-1}$; we say that the word $w_0$ \textit{is obtained from $w$ by externally canceling $t^{-1} t$}. 

We consider the transitive closure of the relation ``being obtained from", that is if $w, w', w''$ are words such that $w'$ is obtained from $w$ by (internally or externally) canceling word $b b^{-1}$ and $w''$ is obtained from $w'$ by canceling word $c c^{-1}$ then we say that \textit{$w''$ is obtained from $w$ by canceling $b b^{-1}$ and $c c^{-1}$}.

It is well known (see \cite{MKS}, Theorem 1.2 of Chapter 1) that given $w \in \mathcal{M}(X \cup X^{-1})$ then $\rho(w)$ is obtained from $w$ by performing all internal cancellations in any order. We have seen in Remark \ref{scope} that $\hat{\rho}(w)$ is obtained from $\rho(w)$ by performing all the external cancellations. So in order to obtain $\hat{\rho}(w)$ from $w$ first we have to carry out all the internal cancellations and afterwards all the external cancellations.

Now the question is: which word do we obtain from $w$ if we do not respect this order and we perform some external cancellation before all internal cancellations have been carried out? Does this word still is $\hat{\rho}(w)$? The answer is that the word obtained is a cyclic permutation of $\hat{\rho}(w)$ and the next theorem proves it.

\begin{theorem} \label{ordCanc} Let $w$ be a word and let $v$ be a cyclically reduced word obtained from $w$ by performing internal and external cancellations in any order. Then $v$ is a cyclic permutation of $\hat{\rho}(w)$. 
\end{theorem}

\begin{proof} Let $u_0 := w, u_1, \dots, u_n := v$ be the sequence of words obtained from $w$ to $q$ by performing internal and external cancellations. 
	
If we perform an internal cancellation we obtain a word with the same reduced form so by Remark \ref{redCycRed} with the same cyclically reduced form. If we perform an external cancellation then by Corollary \ref{permCycRedForm} we obtain a word whose cyclically reduced form is a cyclic permutation. So when going from $u_i$ to $u_{i+1}$ the cyclically reduced form either stays equal or permutes cyclically and this proves the claim.
\end{proof}	

Let us give an interpretation of Theorem \ref{ordCanc} in light of the graph representation of a word given at page \pageref{grafRep}. Let $w$ be a word and let us represent it as a cycle. Let us consider two sequences of cancellations that modify $w$ to two cyclically reduced words $w'$ and $w''$. Then the graph representations of $w'$ and $w''$ are the same possibly with different initial points.

\medskip

Let us give an example when the process described in Theorem \ref{ordCanc} determines a word which is a non-trivial cyclic permutation. Indeed let $w := x x^{-1} y x^{-1}$; if we perform first the external cancellation we obtain the word $x^{-1} y$ which is a non-trivial cyclic permutation of $\hat{\rho}(w) = y x^{-1}$.

\appendix

\section{Identities among relations} \label{AppA}

This section deals with identities among relations. The material until Remark \ref{iarPeiffCons} is already known (see for example \cite{BH}) and is presented here in order to fix the terminology and the notation.

This section is independent from the rest of the paper, except for the notations of the free group on $X$ (denoted $\mathcal{F}(X)$) and of the reduced form of a word ($\rho(w)$ denotes the reduced form of a non-necessarily reduced word $w$) from Section \ref{S1}.

\medskip

Let $a_1, \dots, a_m, r_1, \dots, r_m, b_1, \dots, b_n, s_1, \dots, s_n$ be words such that the equality 
	\begin{equation} \label{idamre0} \rho(a_1 r_1 a_1^{-1} \dots a_m r_m a_m^{-1}) = \rho(b_1 s_1 b_1^{-1} \dots b_n s_n b_n^{-1})\end{equation} 
holds. Then we say that we have an \textit{identity among relations involving $r_1, \dots, r_m, s_1^{-1}, \dots, s_n^{-1}$} that we denote 
	\begin{equation} \label{idamre2} a_1 r_1 a_1^{-1} \centerdot \dots \centerdot a_m r_m a_m^{-1} \equiv b_1 s_1 b_1^{-1} \centerdot \dots \centerdot b_n s_n b_n^{-1}.	\end{equation} 
If $n = 0$, that is the right hand side is 1, then we say that the identity is in \textit{normal form}.

Let us suppose that in the identity (\ref{idamre2}) for some $i \in \{1, \dots, m\}$ we have that $\rho(a_i)$ is a suffix of $\rho(r_i)$ or $\rho(a_i^{-1})$ is a prefix of $\rho(r_i)$. Then we say that (\ref{idamre0}) is \textit{cyclic permutation in the $i$-th left term}. Analogously we define when an identity is cyclic permutation in a term on the right-hand side.

\medskip

Identities among relations are special types of word equations. They arise in the context of group presentations, but we will use them without involving an explicit group presentation. In particular an identity among relations involving $r_1, \dots, r_m$ is an identity among relations for any group presentation having $r_1, \dots, r_m$ as relators. The last claim is obvious if the $r_i$ are basic relators. If some of the $r_i$ are non-basic relators, then the claim follows from Remark \ref{iarRepl}.

\begin{remark} \label{iarRepl} \rm Let us suppose that (\ref{idamre2}) holds and that for some $i$ we have that the reduced form of $r_i$ is equal to the reduced form of $c_1 t_1 c_1^{-1} \dots c_k t_k c_k^{-1}$ for some words $c_1, t_1, \dots c_k, t_k$. Then by replacing in (\ref{idamre2}) the term $a_i r_i a_i^{-1}$ with $d_1 t_1 d_1^{-1} \centerdot \dots \centerdot d_k t_k d_k^{-1}$, where $d_j = a_i c_j$, we obtain an identity among relations involving $s_1, \dots, s_n, t_1, \dots, t_k$ and all the $r_h$ except for $h = i$.
\end{remark}

\begin{definition} \label{}  \rm We say that the identities 
	$$a_2 r_2 a_2^{-1} \centerdot \dots \centerdot a_m r_m a_m^{-1} \equiv a_1 r_1^{-1} a_1^{-1} \centerdot b_1 s_1 b_1^{-1} \centerdot \dots \centerdot b_n s_n b_n^{-1},$$
	$$a_1 r_1 a_1^{-1} \centerdot \dots \centerdot a_{m-1} r_{m-1} a_{m-1}^{-1} \equiv b_1 s_1 b_1^{-1} \centerdot \dots \centerdot b_n s_n b_n^{-1} \centerdot a_m r_m a_m^{-1},$$
	$$b_1 s_1^{-1} b_1^{-1} \centerdot a_1 r_1 a_1^{-1} \centerdot \dots \centerdot a_m r_m a_m^{-1} \equiv b_2 s_2 b_2^{-1} \centerdot \dots \centerdot b_n s_n b_n^{-1}$$
and
	$$a_1 r_1 a_1^{-1} \centerdot \dots \centerdot a_m r_m a_m^{-1} \centerdot b_n s_n b_n^{-1} \equiv b_1 s_1 b_1^{-1} \centerdot \dots \centerdot b_{n-1} s_{n-1} b_{n-1}^{-1}$$
are \textit{1-step equivalent to (\ref{idamre2}}). 

We say that an identity $\iota$ is \textit{equivalent} to an identity $\iota'$ if there exist identities $\iota_1, \dots, \iota_n$ such that $\iota$ is 1-step equivalent to $\iota_1$, $\iota_i$ is 1-step equivalent to $\iota_{i+1}$ for $i \in \{1, \dots, n-1\}$ and $\iota_n$ is 1-step equivalent to $\iota'$.
\end{definition}


Let $\langle \, X \, | \, R \, \rangle$ be a presentation for a group $G$, with $X$ the set of generators and $R$ that of basic relators. We will assume without loss of generality that $R$ contains the inverse of any of its elements and the reduced form of the cyclic permutations of any of its elements. If $r_1, \dots, r_n \in R$ are such that the identity in normal form
	\begin{equation} \label{idamre} a_1 r_1 a_1^{-1} \centerdot \dots \centerdot a_n r_n a_n^{-1} \equiv 1	\end{equation}
holds, then (\ref{idamre}) determines a product of conjugates of basic relators equal to 1 not only in $G$ but also in $\mathcal{F}(X)$ (we recall that $G$ is a quotient of $\mathcal{F}(X)$).

\medskip

\label{H}

In order to formalize these notions we introduce some definitions (we will follow \cite{BH}). Let us set $Y := \mathcal{F}(X) \times R$, let us define the inverse of an element $(a, r) \in Y$ as $(a, r^{-1})$ and let us denote $H$ the free monoid on $Y \cup Y^{-1}$. $H$ is the set of finite sequences of elements of $Y$. We denote an element of $H$ as $[(a_1, r_1), \dots, (a_n, r_n)]$, where $a_i \in \mathcal{F}(X)$ and $r_i \in R$. The trivial element of $H$ is the sequence with zero elements.

Let $h := [(a_1, r_1), \dots, (a_n, r_n)] \in H$ and let $(a, r)$, $(b, s)$ be two consecutive elements $(a_i, r_i), (a_{i+1}, r_{i+1})$ for some $i \in \{1, \dots, n - 1\}$, in particular $a = a_i$, $r = r_i$, $b = a_{i+1}$, $s = r_{i+1}$. We define the following transformations on $h$ that change it to another element of $H$:
\begin{itemize}
	\item [--] a \textit{deletion} deletes in $h$ the elements $(a, r)$, $(b, s)$ if $\rho(a r a^{-1} b s b^{-1}) = 1$;

	\item [--] a \textit{semi-Peiffer deletion} is a deletion where $r^{-1} = s$;

	\item [--] a \textit{Peiffer deletion} is a semi-Peiffer deletion where $a = b$;


	\item [--] an \textit{exchange} replaces in $h$ the elements $(a, r)$, $(b, s)$  either with the pair
		$$(b, s), (\rho(b s^{-1} b^{-1} a), r)$$
	(we call it an \textit{exchange of type A at the $i$-th position} or \textit{exchange of type A-$i$}) or with the pair
		$$(\rho(a r a^{-1} b), s), (a, r)$$
	(we call it an \textit{exchange of type B at the $i$-th position} or \textit{exchange of type B-$i$})).	
\end{itemize}
Deletions and exchanges leave unchanged the $(a_j, r_j)$ for $j \neq i, i+1$.




\begin{remark} \label{} \rm We say that $R$ is \textit{irredundant} if any two elements of $R$ are not conjugate, i.e., for every $r, s \in R$ such that $r \neq s$ there does not exist $a \in \mathcal{F}(X)$ such that $r = \rho(a s a^{-1})$.

By the Lemma at page 174 of \cite{BH} we have that if $R$ is irredundant then a deletion is a semi-Peiffer deletion, that is if $\rho(a r a^{-1} b s b^{-1}) = 1$ for some $(a, r), (b, s) \in Y$ then $r^{-1} = s$.
\end{remark}

\begin{remark} \label{} \rm We say that $R$ is \textit{primary} if no element of $R$ is a proper power, i.e., for every $r \in R$ there does not exist $z \in \mathcal{F}(X)$ and an integer $m$ different than 1 such that $r = \rho(z^m)$.

By Proposition 13 of \cite{BH} we have that if $R$ is irredundant and primary then a deletion is equivalent to some sequence of exchanges plus a Peiffer deletion. In particular if $\rho(a r a^{-1} b s b^{-1}) = 1$ for some $(a, r), (b, s) \in Y$ then by applying some specific exchanges we can transform the pair $(a, r), (b, s)$ into a pair $(c, t), (c, t^{-1})$.
\end{remark}

Given two elements $h_1, h_2 \in H$, we say that \textit{$h_1$ ((semi-)Peiffer) collapses to $h_2$} if $h_2$ can be obtained from $h_1$ by applying ((semi-)Peiffer) deletions and exchanges. 

There is a bijection $\chi$ between $H$ and the set of products of conjugates of elements of $R$ given by associating the element $h = [(a_1, r_1), \dots, (a_n, r_n)] \in H$ with the following product of conjugates of elements of $R$,
	$$a_1 r_1 a_1^{-1} \centerdot \dots \centerdot a_n r_n a_n^{-1}.$$
Also we define a monoid homomorphism $\psi$ from $H$ to $\mathcal{F}(X)$ by $\psi(h) := \rho(a_1 r_1 a_1^{-1} \dots a_n r_n a_n^{-1})$. If $\psi(h) = 1$, that is if $h$ belongs to the kernel of $\psi$, then we say that $h$ determines the identity among relations in normal form (\ref{idamre}). We say that this identity among relations \textit{((semi-)Peiffer) collapses to 1} if $h$ ((semi-)Peiffer) collapses to the trivial element of $H$.

The restriction of $\chi$ to the kernel of $\psi$ determines a bijection with the set of identities among relations in normal form involving elements of $R$.

\begin{remark} \label{iarPeiffCons} \rm We have seen in the introduction to this section that if $r_1$, $\dots$, $r_n$ are relators of a group presentation $\mathcal{P} := \langle \, X \, | \, S \, \rangle$ then an identity among relations involving $r_1, \dots, r_n$ determines an identity among relations for $\mathcal{P}$, that is an identity involving the basic relators of $\mathcal{P}$.
	
By virtue of the Corollary at page 159 of \cite{BH} we have also that if the identity involving $r_1, \dots, r_n$ Peiffer collapses to 1 then also the identity involving basic relators determined by it Peiffer collapses to 1.
\end{remark}

\begin{remark} \label{phi(h)} \rm For $(a, r) \in \mathcal{F}(X) \times R$ we define $\phi\big((a, r)\big) := \big(\rho(a r a^{-1}), r\big)$ and if $h = [(a_1, r_1), \dots, (a_n, r_n)] \in H$ we define 
	$$\phi(h) : = [\phi\big((a_1, r_1)\big), \dots, \phi\big((a_n, r_n)\big)].$$
We observe that $\phi(H) \subset H$. We will usually denote $\phi(h)$ in the following way
	$$\begin{pmatrix}
		\rho(a_1 r_1 a_1^{-1}) \centerdot  & \dots \centerdot & \rho(a_n r_n a_n^{-1})\\
		r_1 	& \dots	& r_n
		\end{pmatrix}.$$

If $\phi\big((a, r)\big) = (\alpha, r)$ and $\phi\big((b, s)\big) = (\beta, s)$ then an exchange of type A replaces in $\phi(h)$ the pair $(\alpha, r), (\beta, s)$ with the pair $(\beta, s), (\beta^{-1} \alpha \beta, r)$; an exchange of type B replaces in $\phi(h)$ the pair $(\alpha, r), (\beta, s)$ with the pair $(\alpha \beta \alpha^{-1}, s), (\alpha, r)$.

By defining exchanges in the above way for the elements of $\phi(H)$ and deletions in the same way as for $H$, we can define the notion of (semi-Peiffer) collapse for the elements of $\phi(H)$. It is easy to see that if $h_1, h_2 \in H$ then $h_1$ (semi-Peiffer) collapses to $h_2$ by means of a certain sequence of operations if and only if $\phi(h_1)$ collapses to $\phi(h_2)$ by means of the same sequence of operations.
\end{remark}

\begin{remark} \label{pcrFrom} \rm Let $R$ be a set of reduced words, that is $R \subset \mathcal{F}(X)$. Let us consider the following operations on the elements of $\mathcal{F}(X)$: reduced product, cyclically reduced product, cyclically reduced form, conjugations, reduced form of cyclic permutations. 

Let $\mathcal{N}$ be the normal closure of $R$ in $\mathcal{F}(X)$; then $\mathcal{N}$ is the subset of $\mathcal{F}(X)$ generated by $R$ and by the above operations. Indeed cyclic permutations and the cyclically reduced form are special cases of conjugations and the cyclically reduced product is obtained by composing the cyclically reduced form with the reduced product.

Let $\sigma$ be a sequence of the above listed operations on the elements of $R$ and let $u \in \mathcal{N}$ be the result of $\sigma$. We will show how to associate with $\sigma$ an element $[(a_1, r_1), \cdots, (a_n, r_n)]$ of $H$ with the property that
	$$\rho(a_1 r_1 a_1^{-1} \cdots a_n r_n a_n^{-1}) = u.$$

\smallskip

- Let us take a sequence of length one. This is an element $r$ of $R$ and we associate with it the element $[(1, r)]$ of $H$. 

We can suppose by induction hypothesis that there is a natural number $k$ such that for each sequence $\sigma$ of length less than $k$ we have associated with $\sigma$ an element of $H$ with the properties specified above.

- Let us given sequences $\sigma, \sigma'$ of length less than $k$ with results respectively $u$ and $u'$. Then by induction hypothesis there exist $r_1, \cdots, r_m$, $s_1, \cdots, s_n \in R$ and $a_1, \cdots, a_m$, $b_1, \cdots, b_n \in \mathcal{F}(X)$ such that we have associated with $\sigma$ an element $[(a_1, r_1), \cdots, (a_m, r_m)] \in H$ such that 
	$$u = \rho(a_1 r_1 a_1^{-1} \cdots a_m r_m a_m^{-1})$$ 
and with $\sigma'$ an element $[(b_1, s_1), \cdots, (b_n, s_n)] \in H$ such that 
	$$u' =\rho(b_1 s_1 b_1^{-1} \cdots b_n s_n b_n^{-1}).$$
Let us consider the sequence $\tau$ having all the operations of $\sigma$ and $\sigma'$ plus the reduced product of $u$ by $u'$. Then we associate with $\tau$ the element
	$$[(a_1, r_1), \cdots, (a_m, r_m), (b_1, s_1), \cdots, (b_n, s_n)] \in H;$$
obviously $\rho(a_1 r_1 a_1^{-1} \cdots a_m r_m a_m^{-1} b_1 s_1 b_1^{-1} \cdots b_n s_n b_n^{-1}) = \rho(u u')$.

- Now let us consider a sequence $\sigma_1$ having all the operations of $\sigma$ plus the conjugation of $u$ by a word $b$. Then we associate with $\sigma_1$ the element $[(c_1, r_1), \cdots, (c_m, r_m)] \in H$ where $c_i = \rho(b a_i)$. Obviously 
	$$\rho(c_1 r_1 c_1^{-1} \cdots c_m r_m c_m^{-1}) = \rho(b u b^{-1}).$$

- Now let us consider a sequence $\sigma_2$ having all the operations of $\sigma$ plus the cyclically reduced form of $u$. The previous cases show how to associate with $\sigma_2$ an element of $H$ with the above properties because by virtue of Remark (\ref{scope}) the cyclically reduced form is a special case of conjugation.

-Now let us consider a sequence $\sigma_3$ having all the operations of $\sigma$ plus the reduced form of a cyclic permutation of $u$. This means that there exist words $u_1, u_2$ such that $u = u_1 u_2$ and that the last operation of $\sigma_3$ is the conjugation of $u$ by either $u_2$ or by $u_1^{-1}$. This implies that we associate with $\sigma_3$ the element $[(c_1, r_1), \cdots, (c_m, r_m)] \in H$ where $c_i$ for $i = 1, \cdots, m$ can be either equal to $\rho(u_2 a_i)$ or to $\rho(u_1^{-1} a_i)$.

- Finally if $\tau$ is the sequence having all the operations of $\sigma$ and $\sigma'$ plus the cyclically reduced product of $u$ by $u'$, then the previous cases show how to associate with $\tau$ an element of $H$ with the properties stated above because the cyclically reduced product is the composition of the reduced product with the cyclically reduced form.
\end{remark}

\begin{remark} \label{seqFromProd} \rm We show how to associate with a product of conjugates of elements of $R$ a sequence of operations on $R$ as described in Remark \ref{pcrFrom}.

Indeed with $a_1 r_1 a_1^{-1} \centerdot \dots \centerdot a_m r_m a_m^{-1}$ we associate the following sequence: conjugation of $r_1$ with $a_1$; conjugation of $r_2$ with $a_2$; $\dots$; conjugation of $r_m$ with $a_m$; reduced product of $a_1 r_1 a_1^{-1}$ by $a_2 r_2 a_2^{-1}$; reduced product of $a_1 r_1 a_1^{-1} a_2 r_2 a_2^{-1}$ by $a_3 r_3 a_3^{-1}$; $\dots$; reduced product of $a_1 r_1 a_1^{-1} \dots  a_{m-1} r_{m-1} a_{m-1}^{-1}$ by $a_m r_m a_m^{-1}$.

In particular, given words $u$ and $v$, we associate with $u*v$ the product $\alpha u \alpha^{-1} \centerdot \alpha v \alpha^{-1}$, where $\alpha$ is such that $u*v = \rho(\alpha u v \alpha^{-1})$ (see Remark \ref{scope}). \end{remark}

\begin{remark} \label{idFrom} \rm Let $u, u' \in \mathcal{N}$ be obtained respectively from sequences $\sigma$ and $\sigma'$ of operations on $R$ as described in Remark \ref{pcrFrom}, in particular in view of Remark \ref{seqFromProd} let $u, u'$ be the reduced forms of products of conjugates of elements of $R$. Let us suppose that $u \sim u'$; we show how to associate with $\sigma$, $\sigma'$ and the equivalence $u \sim u'$ an identity among relations involving elements of $R$.	
	
Indeed the procedure described in Remark \ref{pcrFrom} associates with $\sigma$ and $\sigma'$ elements $h := [(a_1, r_1), \dots, (a_m, r_m)]$ and $h' := [(b_1, s_1), \dots, (b_n, s_n)]$ of $H$ such that $\rho(a_1 r_1 a_1^{-1} \dots a_m r_m a_m^{-1}) = u$ and $\rho(b_1 s_1 b_1^{-1} \dots b_n s_n b_n^{-1}) = u'$.

If $u \sim v$ then $u$ and $v$ are conjugates and thus there exists a word $c$ such that $u = \rho(c v c^{-1})$. We associate with $\sigma$, $\sigma'$ and the equivalence $u \sim v$ the following identity among relations 
	$$a_1 r_1 a_1^{-1} \centerdot \dots \centerdot a_m r_m a_m^{-1} \equiv d_1 s_1 d_1^{-1} \centerdot \dots \centerdot d_n s_n d_n^{-1}$$
where $d_i = c b_i$.	
\end{remark}

The next lemma will be used in Section \ref{S3}

\begin{lemma} \label{collapseh} Let $u, v \in R$ and $\alpha, \beta, \gamma, \delta \in \mathcal{F}(X)$ and let us consider the element $h := [(\alpha, u), (\beta, v), (\gamma, u^{-1}), (\delta, v^{-1})] \in H$.
	
We suppose that there exist words $p, q$ and a natural number $n$ such that $\rho(\alpha u \alpha^{-1}) = \rho(p^{-n} q^{-1})$, $\rho(\beta v \beta^{-1}) = \rho(q p^{n+1})$, $\rho(\gamma u^{-1} \gamma^{-1}) = \rho(p^{n} q)$ and $\rho(\delta v^{-1} \delta^{-1}) = \rho(q^{-1} p^{-n-1})$. 
	
Then $h$ semi-Peiffer collapses to 1 by means of the following sequence of $2n+3$ operations: $n$ exchanges of type B-1; $n$ exchanges of type B-3;\footnote{the $n$ exchanges of type B-1 and those of type B-3 can be made in any order since an exchange of type B-1 commutes with one of type B-3} then an exchange of type A-2; finally two semi-Peiffer deletions.	
\end{lemma}

\begin{proof} In this proof in order to make the notation less cumbersome we will adopt the following convention different from Convention \ref{conv=}: given words $u_1, \dots, u_n$ with the notation $u_1 \dots u_n$ we mean the reduced product of the words $u_1, \dots, u_n$, i.e., their product in $\mathcal{F}(X)$.

For any natural number $k$ we consider the element $\eta_k$ of $H$ defined by 
	$$\eta_k := \begin{pmatrix}
	p^{-n+k} q^{-1} p^{-k} \centerdot & p^k q p^{n-k+1} \centerdot & p^{n-k} q p^k \centerdot & p^{-k} q^{-1} p^{-n+k-1}\\
	u 	& v & u^{-1}	& v^{-1}
	\end{pmatrix}.$$
It is easy to see that all the $\eta_k$ belong to $\phi(H)$. Indeed $\phi(h) = \eta_0$ and in general if we set $h_k := [(p^k\alpha, u), (p^k\beta, v), (p^{-k}\gamma, u^{-1}), (p^{-k}\delta, v^{-1})]$ then $\rho(h_k) = \eta_k$ (we observe that $h_0 = h$). Let us consider for $i \in \{0, 1, 2, 3\}$ the elements $\eta^i_k$ of $H$ defined in the following way:
\begin{enumerate}\setcounter{enumi}{-1}	
	\item $\eta^0_k := \eta_k$;
		
	\item $\eta^1_k := \begin{pmatrix}
		p^{k+1} q p^{n-k} \centerdot & p^{-n+k} q^{-1} p^{-k} \centerdot & p^{n-k} q p^k, u^{-1} \centerdot & p^{-k} q^{-1} p^{-n+k-1}\\
		v 	& u & u^{-1}	& v^{-1}
		\end{pmatrix}$;
		
	\item $\eta^2_k := \begin{pmatrix}
		p^{k+1} q p^{n-k} \centerdot & p^{-n+k} q^{-1} p^{-k} \centerdot & p^{-k-1} q^{-1} p^{-n+k} \centerdot & p^{n-k} q p^k, u^{-1} \\	
		v  & u  & v^{-1}  & u^{-1}
		\end{pmatrix}$;
		
	\item $\eta^3_k := \begin{pmatrix}
		p^{-n+k+1} q^{-1} p^{-k-1} \centerdot & p^{k+1} q p^{n-k} \centerdot & p^{-k-1} q^{-1} p^{-n+k} \centerdot & p^{n-k} q p^k \\	
		u  & v  & v^{-1}  & u^{-1}
		\end{pmatrix}$.	
\end{enumerate}
	
Let $g, g' \in H$. We prove the following facts:
	
\smallskip
	
\noindent (I) if $\phi(g) = \eta^0_k$ and if $g'$ is obtained from $g$ by an exchange of type B-1 then $\phi(g') = \eta^1_k$. Indeed the third and fourth elements of $\eta^0_k$ and $\eta^1_k$ are the same and the first element of $\eta^0_k$ is equal to the second of $\eta^1_k$. It remains to prove that if we set $\alpha := p^{-n+k} q^{-1} p^{-k}$, $\beta := p^k q p^{n-k+1}$ then $p^{k+1} q p^{n-k} = \rho(\alpha \beta \alpha^{-1})$, which is easy to verify.
	
\smallskip
	
\noindent (II) if $\phi(g) = \eta^1_k$ and if $g'$ is obtained from $g$ by an exchange of type B-3 then $\phi(g') = \eta^2_k$. Indeed the first and second elements of $\eta^1_k$ and $\eta^2_k$ are the same and the third element of $\eta^1_k$ is equal to the fourth of $\eta^2_k$. It remains to prove that if we set $\alpha := p^{-n+k} q^{-1} p^{-k}$, $\beta := p^k q p^{n-k+1}$ then $p^{-k-1} q^{-1} p^{-n+k} = \rho(\alpha \beta \alpha^{-1})$, which is easy to verify.

\smallskip
	
\noindent (III) if $\phi(g) = \eta^2_k$ and if $g'$ is obtained from $g$ by an exchange of type B-1 then $\phi(g') = \eta^3_k$. Indeed the third and fourth elements of $\eta^2_k$ and $\eta^3_k$ are the same and the first element of $\eta^2_k$ is equal to the second of $\eta^3_k$. It remains to prove that if we set $\alpha := p^{k+1} q p^{n-k}$, $\beta := p^{-n+k} q^{-1} p^{-k}$ then $p^{-n+k+1} q^{-1} p^{-k-1} = \rho(\alpha \beta \alpha^{-1})$, which is easy to verify.
	
\smallskip
	
\noindent (IV) if $\phi(g) = \eta^3_k$ and if $g'$ is obtained from $g$ by an exchange of type B-3 then $\phi(g') = \eta^0_{k+1}$. Indeed the first and second elements of $\eta^3_k$ and $\eta^0_{k+1}$ are the same and the third element of $\eta^3_k$ is equal to the fourth of $\eta^0_{k+1}$. It remains to prove that if we set $\alpha := p^{-k-1} q^{-1} p^{-n+k}$, $\beta := p^{n-k} q p^k$ then $p^{n-k-1} q p^{k+1} = \rho(\alpha \beta \alpha^{-1})$, which is easy to verify.

\medskip
	
For each natural number $i$ we define $g^k_i \in H$ in the following way: $g^k_0 := h_k$; if $i>0$ and $i$ is odd then $g^k_i$ is the element of $H$ obtained from $g^k_{i-1}$ after an exchange of type B-1; if $i>0$ and $i$ is even then $g^k_i$ is the element of $H$ obtained from $g^k_{i-1}$ after an exchange of type B-3.

We now prove that for every natural number $m$ and for every $i \in \{0, 1, 2, 3\}$ we have that $\phi(g_{4m+i}^k) = \eta^i_{k+m}$. If $m = 0$ then the claim follows by definition of $g$. Let $m > 0$ and the claim be true for $m - 1$; then we will prove that $\phi(g_{4m+i}^k) = \eta^i_{k+m}$ for $i = 0$, which will imply the claim for the other $i$'s as well. Indeed we have that $\phi(g_{4m}^k) = \phi(g_{4(m - 1) + 4}^k)$. We have that $g_{4(m - 1) + 4}^k$ is obtained from $g_{4(m - 1) + 3}^k$ by an exchange of type B-3; also by induction hypothesis we have that $\phi(g_{4(m - 1) + 3}^k) = \eta^3_{k+m-1}$. By what seen in (IV) we have that $\phi(g_{4(m - 1) + 4}^k) = \eta^0_{k+m}$, proving the claim.
	
From now we will assume that $k=0$ (we have that $h_0 = h$), thus	
	$$\phi(h) = \eta_0^0 = 
	\begin{pmatrix}
	p^{-n} q^{-1} \centerdot & q p^{n+1} \centerdot & p^{n} q \centerdot & q^{-1} p^{-n-1}\\
	u 	& v & u^{-1}	& v^{-1}
	\end{pmatrix}.$$
For every $j$ we will denote $g_j$ the element $g^0_j$. We have that $g_{2n}$ is the element obtained from $h$ after $n$ times the pair of exchanges B-1, B-3.
	
Let $n$ be even; then $n = 2m$ for some $m$ and $\phi(g_{2n}) = \phi(g_{4m}) = \eta^0_{m}$. Since $n - m = m$ then 
	$$\eta^0_{m} = 
	\begin{pmatrix}
	p^{-m} q^{-1} p^{-m} \centerdot & p^m q p^{m+1} \centerdot & p^{m} q p^m \centerdot & p^{-m} q^{-1} p^{-m-1}\\
	u 	& v & u^{-1}	& v^{-1}
	\end{pmatrix}.$$
	
If we apply an exchange of type A-2 to $g_{2n}$ we obtain an element $g' \in H$ such that	
	$$\phi(g') = \begin{pmatrix}
	p^{-m} q^{-1} p^{-m} \centerdot & p^{m} q p^m \centerdot & p^{m+1} q p^m \centerdot & p^{-m} q^{-1} p^{-m-1}\\
	u & u^{-1}	& v & v^{-1}
	\end{pmatrix}$$	
because if it is easy to see that if we set $\alpha := p^m q p^{m+1}$ and $\beta := p^{m} q p^m$ then $p^{m+1} q p^m = \rho(\beta^{-1} \alpha \beta)$. Since the first and second elements of $\phi(g')$ are inverse, as well as the third and fourth, we have that $\phi(g')$ reduces to 1 by means of two deletions. The claim for $n$ even follows by means of Remark \ref{phi(h)}

Let $n$ be odd; then $n = 2m + 1$ for some $m$ and $\phi(g_{2n}) = \phi(g_{4m+2}) = \eta^2_{m}$. Since $n - m = m +1$ then 
	$$\eta^2_{m} = \begin{pmatrix}
	p^{m+1} q p^{m+1} \centerdot & p^{-m-1} q^{-1} p^{-m} \centerdot & p^{-m-1} q^{-1} p^{-m-1} \centerdot & p^{m+1} q p^m\\
	v 	&  u & v^{-1} & u^{-1}
	\end{pmatrix}.$$
	
If we apply an exchange of type A-2 to $g_{2n}$ we obtain an element $g' \in H$ such that
	$$\phi(g') = \begin{pmatrix}
	p^{m+1} q p^{m+1} \centerdot & p^{-m-1} q^{-1} p^{-m-1} \centerdot & p^{-m} q^{-1} p^{-m-1} \centerdot & p^{m+1} q p^m\\
	v 	& v^{-1} & u	& u^{-1}
	\end{pmatrix}$$	
because if it is easy to see that if we set $\alpha := p^{-m-1} q^{-1} p^{-m}$ and $\beta := p^{-m-1} q^{-1} p^{-m-1}$ then $p^{-m} q^{-1} p^{-m-1} = \rho(\beta^{-1} \alpha \beta)$. Since the first and second elements of $\phi(g')$ are inverse, as well as the third and fourth, we have that $\phi(g')$ reduces to 1 by means of two deletions. The claim for $n$ odd follows by means of Remark \ref{phi(h)}  				
\end{proof}

\section{Some technical results} \label{AppB}

This section depends only on Sections \ref{S1} and \ref{S2} and can be read independently of Appendix \ref{AppA}.

\begin{lemma} \label{complic} Let $w$ be a non-empty cyclically reduced word and let $b$ be a reduced word. Then there exist words $w_1, w_2, b_1$ and a natural number $n$ such that $w = w_1 w_2$, $\rho(b w b^{-1}) = b_1 w_2 w_1 b_1^{-1}$ and at least one of the two holds
	\begin{enumerate} [(1)]
		\item $w b^{-1}$ is reduced, $w_2 \neq 1$ and $b = b_1 w_1^{-1} w^{-n}$;
		
		\item $b w$ is reduced, $w_1 \neq 1$ and $b = b_1 w_2 w^n$.		
	\end{enumerate}
\end{lemma}


\begin{proof} If $b w b^{-1}$ is reduced then both (1) and (2) hold by taking $n = 0$ and $w_1 = 1$ in (1) and $w_2 = 1$ in (2).
	
Now let $b w b^{-1}$ be not reduced; then by Remark \ref{uvw} at least one between $b w$ and $w b^{-1}$ is non-reduced because $w \neq 1$. Indeed exactly one of the two is non-reduced because otherwise $w$ would not be cyclically reduced.
	
First we show that if we have proved (1) then (2) follows easily by taking inverses. Indeed if $b w$ is reduced then $w^{-1} b^{-1}$ is reduced. From (1) it follows that there exist words $w_1, w_2, b_1$ and a natural number $n$ such that $w^{-1} = w_2^{-1} w_1^{-1}$, $\rho(b w^{-1} b^{-1}) = b_1 w_1^{-1} w_2^{-1} b_1^{-1}$, $w_1 \neq 1$ and $b = b_1 w_2 w^n$. Thus we have that $\rho(b w b^{-1}) = b_1 w_2 w_1 b_1^{-1}$ and the claim is proved.
	
Now we assume that $w b^{-1}$ is reduced, that is $b w$ is not reduced and we prove (1) by induction on $|b|$. Let $|b| = 1$; then there exists a word $w_2$ such that $w = b^{-1} w_2$, thus 
	$$\rho(b w b^{-1}) = \rho(b b^{-1} w_2 b^{-1}) = \rho(w_2 b^{-1}) = w_2 b^{-1}$$
and the claim follows by setting $b_1 := 1$, $w_1 := b^{-1}$, $n = 0$.

Let us assume that $|b| > 0$ and the claim be true for every $b'$ of length less than $|b|$.
	
We observe that a factorization of the form $b = b_1 w_1^{-1} w^{-n}$ always exists, for example by taking $w_1 = 1$ and $n = 0$. We also observe that once we take a prefix $w_1$ of $w$, then there exists one and only one $w_2$ such that $w = w_1 w_2$.

It is obvious that there exists a factorization of the above form with $n$ and $|w_1|$ maximal. Let us take that factorization; since $n$ is maximal then $w_1 \neq w$ and thus $w_2 \neq 1$; since $|w_1|$ is maximal then $b_1 w_2$ is reduced.

We have that 
	$$\rho(b w b^{-1}) = \rho(b_1 w_1^{-1} w^{-n} w w^n w_1 b_1^{-1}) =$$
	$$\rho(b_1 w_1^{-1} w w_1 b_1^{-1})  = \rho(b_1 w_1^{-1} w_1 w_2 w_1 b_1^{-1}) = \rho(b_1 w_2 w_1 b_1^{-1}).$$
Let $b_1 = 1$; then since $w$ is cyclically reduced then $w_2 w_1$ is reduced, thus
	$$\rho(b_1 w_2 w_1 b_1^{-1}) = \rho(w_2 w_1) = w_2 w_1,$$
proving the claim.
	
Let $b_1 \neq 1$; we prove that if $w_1 \neq 1$ then $b_1 w_2 w_1 b_1^{-1}$ is reduced, proving the claim. Indeed $b_1 w_2$ and $w_2 w_1$ are reduced by what seen above and since $w_2 \neq 1$ then $b_1 w_2 w_1$ is reduced by Remark \ref{uvw}. If $b_1 w_2 w_1 b_1^{-1}$ is not reduced then there must be cancellation in $w_1 b_1^{-1}$ since $w_1 \neq 1$. If this were the case then there would exist words $w_3, x, b_2$ with $x \neq 1$ such that $w_1 = w_3 x$ and $b_1^{-1} = x^{-1} b_2^{-1}$. But this implies that $b_1 = b_2 x$ and $w_1^{-1} = x^{-1} w_3^{-1}$, therefore $b = b_2 x x^{-1} w_3^{-1} w^{-n}$ and there would be a non-trivial cancellation in $b$, contradicting the fact that $b$ is reduced.

Finally let $b_1 \neq 1$ and $w_1 = 1$; then $w_2 = w$, $\rho(b w b^{-1}) = \rho(b_1 w b_1^{-1})$ and $b = b_1 w^{-n}$.

First we assume that $n \neq 0$; then we have that $|b_1| < |b|$. If $b_1 w b_1^{-1}$ is reduced (in particular if $|b_1| = 0$), then $\rho(b w b^{-1}) = b_1 w b_1^{-1}$ and the claim is true. Otherwise by applying the induction hypothesis we have that there exist words $w_3, w_4, b_2$ and a natural number $m$ such that $w = w_3 w_4$, $w_4 \neq 1$, $b_1 = b_2 w_3^{-1} w^{-m}$ and $\rho(b_1 w b_1^{-1}) = b_2 w_4 w_3 b_2^{-1}$. Thus 
	$$b = b_2 w_3^{-1} w^{-m} w^{-n} = b_2 w_3^{-1} w^{-m -n}$$
and since $\rho(b w b^{-1}) = \rho(b_1 w b_1^{-1})$ the claim is proved.

Now we show that if $n = 0$ we have a contradiction. Indeed we have that $b = b_1$ and therefore $b w$ is reduced, but we have assumed that $b w$ is not reduced.	
\end{proof}

We observe that the classification given in Lemma \ref{shirv} has been assumed but not proved in (\cite{Ivanov06}, pag. 1564). Cases 1), 2) and 3) of Lemma \ref{shirv} correspond respectively to cases F2, F1 and F3 of (\cite{Ivanov06}, pag. 1564). We also observe that the proof of Lemma \ref{shirv} we give here follows closely the one we gave in \cite{PhDThs}.

\begin{lemma} \label{shirv} Let $u$ and $v$ be reduced words such that $u \neq v^{-1}$. Then one of the following holds:
	\begin{enumerate} [1)]
		\item there exist words $u_1, a, s$ such that $u = u_1 a$, $v = a^{-1} s (u*v) s^{-1} u_1^{-1}$ and $\rho(uv) = u_1 s (u*v) s^{-1} u_1^{-1}$;
		
		\item there exist non-empty words $c_1, c_2$ and words $t, a$ such that $u*v = c_1 c_2$, $u = t c_1 a$, $v = a^{-1} c_2 t^{-1}$, $\rho(uv) = t c_1 c_2 t^{-1}$, $\rho(v u) = a^{-1} c_2 c_1 a$ and $v*u = c_2 c_1$; 
		
		\item there exist words $v_1, s, a$ such that $u = v_1^{-1} s (u*v) s^{-1} a$, $v = a^{-1} v_1$ and $\rho(uv) = v_1^{-1} s (u*v) s^{-1} v_1$. 
	\end{enumerate}
\end{lemma} 

\noindent \textbf{Observation} We recall that by Convention \ref{conv=} the equalities in 1), 2) and 3) of Lemma \ref{shirv} (like all the equalities of this paper) are equalities in the free monoid $\mathcal{M}(X \cup X^{-1})$ and not only equalities in $\mathcal{F}(X)$. For instance in 1) we have that $v$ can be factored as $a^{-1} s (u*v) s^{-1} u_1^{-1}$, so the word $a^{-1} s (u*v) s^{-1} u_1^{-1}$ is reduced since it is equal to $v$ which is reduced. 

\smallskip

\begin{proof} \textbf{of Lemma \ref{shirv}.} As in Remark \ref{reducprod}, let $u_1, v_1, a\in \mathcal{F}(X)$ be such that $u = u_1 a$, $v = a^{-1} v_1$ and $\rho(u v) = u_1 v_1$. As in Remark \ref{scope} let $t \in \mathcal{F}(X)$ be such that $u_1 v_1 = t (u*v) t^{-1}$. 
	
If $u*v = 1$ then $u = v^{-1}$ by Remark \ref{jarem}. We can therefore assume that $u*v \neq 1$. Since $u_1 v_1 = t (u*v)  t^{-1}$, three cases are possible: 
\begin{enumerate}
	\item $u_1$ is a prefix of $t$;
		
	\item $u_1$ is a prefix of $t (u*v)$ but not of $t$;
		
	\item $u_1$ is not a prefix of $t(u*v)$.
\end{enumerate}
	
Let us examine the three cases.	
	
\medskip
	
	\textbf{1.} \begin{tabular}{|c|c|}
		\hline
		\rule{0pt}{2.3ex}
		$u_1$ & $\,\,\,\,\,\,\,\,\,\,\,\,\,\,\, v_1 \,\,\,\,\,\,\,\,\,\,\,\,\,\, \hspace{0.5 mm}$ \\  
		\hline
	\end{tabular}
	
	\hspace{3.9mm} \begin{tabular}{|c|c|c|}
		\hline 
		\rule{0pt}{2.3ex}
		$\,\,\,\,\,\, t \,\,\,\,\,\,$ & $u*v$ & $t^{-1}$ \\
		\hline
	\end{tabular}
	
\medskip
	
\noindent There exists a word $s$ such that $t = u_1 s$ and $v_1 = s (u*v) t^{-1}$, thus $t^{-1} = s^{-1} u_1^{-1}$ and then $v_1 = s (u*v) s^{-1} u_1^{-1}$, thus $\rho(uv) = u_1 s (u*v) s^{-1} u_1^{-1}$. Since $v = a^{-1} v_1$ then $v = a^{-1} s (u*v) s^{-1} u_1^{-1}$.

\medskip

	\textbf{2.} \begin{tabular}{|c|c|}
		\hline
		\rule{0pt}{2.3ex}
		$\,\,\, u_1 \,\,\,$ & $\,\,\,\,\,\, v_1 \,\,\,\,\, \hspace{0.5 mm}$ \\  
		\hline
	\end{tabular}
	
	\hspace{3.9mm} \begin{tabular}{|c|c|c|}
		\hline 
		\rule{0pt}{2.3ex}
		$t$ & $u*v$ & $t^{-1}$ \\
		\hline
	\end{tabular}
	
	\medskip
	
\noindent There exist words $c_1, c_2$ such that $u_1 = t c_1$, $u*v = c_1 c_2$ and $v_1 = c_2 t^{-1}$, therefore $\rho(uv) = t c_1 c_2 t^{-1}$. Since $u = u_1 a$ then $u = t c_1 a$; since $v = a^{-1} v_1$ then $v = a^{-1} c_2 t^{-1}$.

\medskip

	\textbf{3.} \begin{tabular}{|c|c|}
		\hline
		\rule{0pt}{2.3ex}
		$\,\,\,\,\,\,\,\,\,\,\,\, u_1 \,\,\,\,\,\,\,\,\,\,\, \hspace{0.5 mm}$ & $v_1$ \\  
		\hline
	\end{tabular}
	
	\hspace{3.9mm} \begin{tabular}{|c|c|c|}
		\hline 
		\rule{0pt}{2.3ex}
		$t$ & $u*v$ & $\,\,\, t^{-1} \,\,\,$ \\
		\hline
	\end{tabular}
	
	\medskip
	
\noindent There exists a word $s$ such that $u_1 = t (u*v) s^{-1}$ and $t^{-1} = s^{-1} v_1$, thus $t = v_1^{-1} s$ and then $u_1 = v_1^{-1} s (u*v) s^{-1}$ and therefore $\rho(uv) = v_1^{-1} s (u*v) s^{-1} v_1$. Since $u = u_1 a$, then $u = v_1^{-1} s (u*v) s^{-1} a$.
	
\medskip
	
Case 2 will reduce to case 1 if $c_1 = 1$ and it will reduce to case 3 if $c_2 = 1$. Therefore we can assume that in case 2, $c_1$ and $c_2$ are non-empty words. This implies that in case 2 we have that
	$$\rho(v u) = \rho(a^{-1} c_2 c_1 a) = a^{-1} c_2 c_1 a,$$
where the last equality follows from Remark \ref{uvw}. This implies that $v*u = c_2 c_1$.
\end{proof}

\begin{remark} \label{} \rm The case where there is no cancellation between $u$ and $v$ in their cyclically reduced product corresponds to case 2) of Lemma \ref{shirv}. In that case $a = t = 1$.
\end{remark}

The graphical interpretation of Lemma \ref{shirv} is the following. Case 1) corresponds to Figure \ref{fig:case1Lem}: the internal circle represents the word $u$, which is completely canceled in the cyclically reduced product of $u$ by $v$. If $s = 1$ then the edge labeled by $s$ reduces to a point and the internal circle is tangent to the external cycle labeled by $w$.

Case 2) corresponds to Figure \ref{fig:case2Lem}. If $t = 1$ and $a \neq 1$ then the initial and final points of the edge labeled by $t$ coincide and that point will be also the final point of the edge labeled by $a$. If $a = 1$ and $t \neq 1$ then the situation is the same as above with the roles of $t$ and $a$ swapped. If both $t$ and $a$ are equal to 1 then the shape will be different, in particular it will consist of two circles labeled by $c_1$ and $c_2$ sharing one vertex.

Case 3) corresponds to Figure \ref{fig:case1Lem} but with the roles of $u$ and $v$ swapped.

\begin{figure}[h!]
	\begin{minipage}[b]{0.4\textwidth}
		\begin{picture}(180, 120) (95, -20)
		
			\put(163, 40){\oval(80, 80)}
		
			\put(145, 88){\vector(1, 0){4}}
			\put(145, 70){\oval(55, 35)[tl]}
		
			\put(180, 85){$w$}
			
			\put(123, 40){\line(1, 0){32}}
		
			\put(123, 40){\circle{5}}
			\put(123, 40){\circle*{3}}
		
			\put(133, 40){\vector(-1, 0){4}}
		 	\put(144, 42){$s$}
		
			\put(155, 40){\circle*{3}}
			\put(195, 40){\circle*{3}}

			\put(175, 40){\circle{38}}
		
			\put(175, 56){\vector(1, 0){4}}
			\put(175, 42){\oval(27, 27)[tl]}
		
			\put(170, 63){$a$}
			\put(170, 13){$u_1$}
		\end{picture}
	\caption{case 1) with $w := u*v$}
	\label{fig:case1Lem}
	\end{minipage}
	\hfill
	\begin{minipage}[b]{0.4\textwidth}
		\begin{picture}(180, 120)(95, -30)
		
			\put(163, 40){\oval(60, 60)}
			\put(163, 10){\line(0, 1){60}}
		
			\put(133, 40){\vector(0, 1){4}}
		 	\put(120, 42){$c_1$}
			
			\put(163, 10){\circle{5}}
			\put(163, 10){\circle*{3}}
			\put(163, 40){\circle*{3}}
			\put(163, 70){\circle*{3}}
		
			\put(163, 50){\vector(0, -1){4}}
			\put(163, 20){\vector(0, -1){4}}
		
			\put(155, 60){$a$}
			\put(155, 30){$t$}

			\put(193, 44){\vector(0, -1){4}}
		 	\put(198, 42){$c_2$}		
		\end{picture}
		\vspace*{-3mm}
		\caption{case 2)}
		\label{fig:case2Lem}
	\end{minipage}
\end{figure}

\begin{lemma} \label{shirv4} Let $u$ and $v$ be reduced words such that $u \neq v^{-1}$ and let $d$ be a cyclic permutation of $u*v$. Then there exists a pair of words $p$ and $q$ such that one of them is a cyclic permutation of $u$ and the other of $v$ and such that one of the following two cases holds: \begin{enumerate} [(a)]
		\item $q = p^{-1} r c_1 c_2 r^{-1}$ for words $r, c_1, c_2$ such that $p*q = c_1 c_2$ and $d = c_2 c_1$; moreover $p*q = u*v$; finally if $d = u*v$ then $c_2 = 1$;

		
		\item $p = e_2 b$ and $q = b^{-1} e_3 e_1$ for words $b, e_1, e_2, e_3$ such that $d = e_1 e_2 e_3$ and $e_2, e_3 e_1 \neq 1$; moreover either $p*q = u*v$ or $q*p = u*v$; finally if $d = u*v$ then $e_1 = 1$.		
	\end{enumerate} 
\end{lemma}

\begin{proof} We show that the claim is true for the three cases of Lemma \ref{shirv}.
	
\textbf{Case 1.} We have that $u = u_1 a$, $v = a^{-1} r (u*v) r^{-1} u_1^{-1}$. If we set $p := a u_1$ and $q := u_1^{-1} a^{-1} r (u*v) r^{-1} =  p^{-1} r (u*v) r^{-1}$, then $p$ is a cyclic permutation of $u$, $q$ is a cyclic permutation of $v$ and 
	$$p*q = \hat{\rho}(pq) = \hat{\rho}(a u_1 u_1^{-1} a^{-1} r (u*v) r^{-1}) = \hat{\rho}(r (u*v) r^{-1}) = u*v,$$
where we have used the fact that $r (u*v) r^{-1}$ is reduced since it is a subword of $v$, which is reduced.

Moreover since $d$ is a cyclic permutation of $u*v$ there exist words $c_1, c_2$ such that $u*v = c_1 c_2$ and $d = c_2 c_1$. Finally if $d = u*v$ then we can take $c_2 = 1$, so case (a) holds.

\smallskip

\textbf{Case 2.} We have that there exist non-empty words $c_1, c_2$ such that $u*v = c_1 c_2$ and $u = t c_1 a$, $v = a^{-1} c_2 t^{-1}$. If we set $u' := c_1 a t$ and $v' := t^{-1} a^{-1} c_2$, then $u'$ is a cyclic permutation of $u$ and $v'$ a cyclic permutation of $v$. 

If $d = u*v$ then we set $e_2 := c_1$, $e_3 := c_2$, $b := at$, $p := u'$, $q := v'$ and case (b) holds. 

Let us study the general case. Since $d$ is a cyclic permutation of $u*v$, there exist words $d_1, d_2$ such that $u*v = d_1 d_2$ and $d = d_2 d_1$. Since $c_1 c_2 = d_1 d_2$ then by Remark \ref{LeviLemma} there exists a word $x$ such that either $c_1 = d_1 x$ and $d_2 = x c_2$ or $d_1 = c_1 x$ and $c_2 = x d_2$.

\smallskip

\textbf{-- Case 2.1:} $c_1 = d_1 x$ and $d_2 = x c_2$. Then $u' = d_1 x a t$.
	
Let us set $e_1 := x$, $e_2 := c_2$, $e_3 := d_1$. Then $u' = e_3 e_1 a t$, $v' = t^{-1} a^{-1} e_2$ and $c_1 = e_3 e_1$. Moreover $u*v = c_1 c_2 = e_3 e_1 e_2$ and $e_1 e_2 e_3 = x c_2 d_1 = d_2 d_1 = d$. 
	
Set $b := t^{-1} a^{-1}$; then $u' = e_3 e_1 b^{-1}$ and $v' = b e_2$ and we set $p := v'$, $q := u'$. Finally $q*p = \hat{\rho}(qp) = \hat{\rho}(e_3 e_1 b^{-1} b e_2) = e_3 e_1 e_2 = u*v$, thus case (b) holds.

\smallskip

\textbf{-- Case 2.2:} $d_1 = c_1 x$ and $c_2 = x d_2$. Then $v' = t^{-1} a^{-1} x d_2$. 
	
Let us set $e_1 := d_2$, $e_2 := c_1$, $e_3 := x$. Then $u' := e_2 a t$, $v' = t^{-1} a^{-1} e_3 e_1$ and $c_2 = e_3 e_1$. Moreover $u*v = c_1 c_2 = e_2 e_3 e_1$ and $e_1 e_2 e_3 = d_2 c_1 x = d_2 d_1 = d$. 
	
Set $b := a t$, $p := u'$ and $q := v'$; then $p = e_2 b$, $q = b^{-1} e_3 e_1$ and $p*q = \hat{\rho}(pq) = \hat{\rho}(e_2 b b^{-1} e_3 e_1) = e_2 e_3 e_1 = u*v$, thus case (b) holds.

\smallskip
	
\textbf{Case 3.} We have that $u = v_1^{-1} r (u*v) r^{-1} a$ and $v = a^{-1} v_1$.  If we set $p := v_1 a^{-1}$ and $q := a v_1^{-1} r (u*v) r^{-1} = p^{-1} r (u*v) r^{-1}$, then $p$ is a cyclic permutation of $v$, $q$ is a cyclic permutation of $u$ and 
	$$p*q = \hat{\rho}(pq) = \hat{\rho}(v_1 a^{-1} a v_1^{-1} r (u*v) r^{-1}) = \hat{\rho}(r (u*v) r^{-1}) = u*v,$$
where we have used the fact that $r (u*v) r^{-1}$ is reduced since it is a subword of $u$, which is reduced.

Moreover since $d$ is a cyclic permutation of $u*v$ there exist words $c_1, c_2$ such that $u*v = c_1 c_2$ and $d = c_2 c_1$. Finally if $d = u*v$ then we can take $c_2 = 1$, so case (a) holds. 
\end{proof} 

\smallskip

\begin{lemma} \label{ax=b} Let $X$ have at least two elements and let $u, w$ be reduced non-empty words. Then there exists a word $s$ of length at most 2 such that for every $n \in \mathbb{N}\setminus\{0\}$ the word $u s^n w s^{-n}$ is cyclically reduced.
\end{lemma} 

\begin{proof} It is enough to prove that there exists $s$ such that $|s| \leq 2$ and $u s w s^{-1}$ is cyclically reduced. Indeed since $s$ is reduced and has length at most 2, then it is cyclically reduced, so $s^n$ and $s^{-n}$ are cyclically reduced. Moreover the first and last letters of $s^n$ [respectively of $s^{-n}$] are the same as those of $s$ [respectively of $s^{-1}$], so if $u s w s$ is cyclically reduced then so is $u s^n w s^{-n}$.

\smallskip

We have that there exist $a, b, c, d \in X \cup X^{-1}$ and words $u_1, u_2, w_1, w_2 \in \hat{\mathcal{F}}(X)$ such that $u = a u_1 = u_2 b$, $w = c w_1 = w_2 d$. The words $u_1, u_2, w_1, w_2$ are necessarily reduced and they can be equal to 1 (in case $u$ or $w$ is a single letter).

\smallskip

First we assume that $u$ is cyclically reduced; this implies that $a \neq b^{-1}$. We split the proof in two different cases: 1) $a = b$; 2) $a \neq b$.

\smallskip

\textbf{Case 1.} Then $u = a u_1 = u_2 a$. Since $X$ has at least two elements, there exists $x \in X \cup X^{-1}$ such that $x \neq a, a^{-1}$. This implies in particular that the words $x a$, $x a^{-1}$, $a x$, $a^{-1} x$ are (cyclically) reduced. We split Case 1 into five subcases.

\smallskip

-- \textbf{Case 1.1}: $x \neq c^{-1}, d$. Then we can take $s = x$. Indeed $u s w s^{-1}$ is cyclically reduced because: the last letter of $u$ is $a$, the first letter of $s$ is $x$ and $a \neq x^{-1}$; the last letter of $s$ is $x$, the first letter of $w$ is $c$ and $x \neq c^{-1}$; the last letter of $w$ is $d$, the first letter of $s^{-1}$ is $x^{-1}$ and $d \neq x$; the last letter of $s^{-1}$ is $x^{-1}$, the first letter of $u$ is $a$ and $x \neq a$.

\smallskip

-- \textbf{Case 1.2}: $x = c^{-1}, x \neq d^{-1}$, thus $w = x^{-1} w_1 = w_2 d$. Then we can take $s = x^{-1}$. Indeed $u s w s^{-1}$ is cyclically reduced because: the last letter of $u$ is $a$, the first letter of $s$ is $x^{-1}$ and $a \neq x$; the last letter of $s$ is $x^{-1}$ and the first letter of $w$ is also $x^{-1}$; the last letter of $w$ is $d$, the first letter of $s^{-1}$ is $x$ and $d \neq x^{-1}$; the last letter of $s^{-1}$ is $x$, the first letter of $u$ is $a$ and $x^{-1} \neq a$.

\smallskip

-- \textbf{Case 1.3}: $x = c^{-1} = d^{-1}$, thus $w = x^{-1} w_1 = w_2 x^{-1}$. Then we can take $s = x a^{-1}$. Indeed $u s w s^{-1}$ is cyclically reduced because: the last letter of $u$ is $a$, the first letter of $s$ is $x$ and $a^{-1} \neq x$; the last letter of $s$ is $a^{-1}$, the first letter of $w$ is $x^{-1}$ and $a \neq x^{-1}$; the last letter of $w$ is $x^{-1}$, the first letter of $s^{-1}$ is $a$ and $x \neq a$; the last letter of $s^{-1}$ is $x^{-1}$, the first letter of $u$ is $a$ and $x \neq a$.


\smallskip

-- \textbf{Case 1.4}: $x = d, x \neq c$, thus $w = c w_1 = w_2 x$. Then we can take $s = x^{-1}$. Indeed $u s w s^{-1}$ is cyclically reduced because: the last letter of $u$ is $a$, the first letter of $s$ is $x^{-1}$ and $a \neq x$; the last letter of $s$ is $x^{-1}$, the first letter of $w$ is $c$ and $x \neq c$; the last letter of $w$ is $x$ and the first letter of $s^{-1}$ is also $x$; the last letter of $s^{-1}$ is $x$, the first letter of $u$ is $a$ and $x^{-1} \neq a$.

\smallskip

-- \textbf{Case 1.5}: $x = d = c$, thus $w = x w_1 = w_2 x$. Then we can take $s = x^{-1} a^{-1}$. Indeed $u s w s^{-1}$ is cyclically reduced because: the last letter of $u$ is $a$, the first letter of $s$ is $x^{-1}$ and $a^{-1} \neq x^{-1}$; the last letter of $s$ is $a^{-1}$, the first letter of $w$ is $x$ and $a \neq x$; the last letter of $w$ is $x$, the first letter of $s^{-1}$ is $a$ and $x^{-1} \neq a$; the last letter of $s^{-1}$ is $x$, the first letter of $u$ is $a$ and $x^{-1} \neq a$.


\medskip

\textbf{Case 2.} We split Case 2 into five subcases.

\smallskip

-- \textbf{Case 2.1}: $b \neq c^{-1}, d$. Then we can take $s = b$. Indeed $u s w s^{-1}$ is cyclically reduced because: the last letter of $u$ is $b$ and the first letter of $s$ is also $b$; the last letter of $s$ is $b$, the first letter of $w$ is $c$ and $b \neq c^{-1}$; the last letter of $w$ is $d$, the first letter of $s^{-1}$ is $b^{-1}$ and $d^{-1} \neq b$; the last letter of $s^{-1}$ is $b^{-1}$, the first letter of $u$ is $a$ and $b \neq a$.

\smallskip

-- \textbf{Case 2.2}: $b = c^{-1}, a \neq d^{-1}$, thus $w = b^{-1} w_1 = w_2 d$. Then we can take $s = a^{-1}$. Indeed $u s w s^{-1}$ is cyclically reduced because: the last letter of $u$ is $b$, the first letter of $s$ is $a^{-1}$ and $b \neq a$; the last letter of $s$ is $a^{-1}$, the first letter of $w$ is $b^{-1}$ and $a \neq b^{-1}$; the last letter of $w$ is $d$, the first letter of $s^{-1}$ is $a$ and $d \neq a^{-1}$; the last letter of $s^{-1}$ is $a$ and the first letter of $u$ is also $a$.



\smallskip

-- \textbf{Case 2.3}: $b = c^{-1}, a = d^{-1}$, thus $w = b^{-1} w_1 = w_2 a^{-1}$. Then we can take $s = b a$. Indeed $u s w s^{-1}$ is cyclically reduced because: the last letter of $u$ is $b$, the first letter of $s$ is also $b$; the last letter of $s$ is $a$, the first letter of $w$ is $b^{-1}$ and $a \neq b$; the last letter of $w$ is $a^{-1}$ and the first letter of $s^{-1}$ is also $a^{-1}$; the last letter of $s^{-1}$ is $b^{-1}$, the first letter of $u$ is $a$ and $b \neq a$.

\smallskip

-- \textbf{Case 2.4}: $b = d, a \neq c$, thus $w = c w_1 = w_2 b$. Then we can take $s = a^{-1}$. Indeed $u s w s^{-1}$ is cyclically reduced because: the last letter of $u$ is $b$, the first letter of $s$ is $a^{-1}$ and $b \neq a$; the last letter of $s$ is $a^{-1}$, the first letter of $w$ is $c$ and $a \neq c$; the last letter of $w$ is $b$, the first letter of $s^{-1}$ is $a$ and $b^{-1} \neq a$; the last letter of $s^{-1}$ is $a$ and the first letter of $u$ is also $a$.

\smallskip

-- \textbf{Case 2.5}: $b = d, a = c$, thus $w = a w_1 = w_2 b$. Then we can take $s = b a$. Indeed $u s w s^{-1}$ is cyclically reduced because: the last letter of $u$ is $b$ and the first letter of $s$ is also $b$; the last letter of $s$ is $a$ and the first letter of $w$ is also $a$; the last letter of $w$ is $b$, the first letter of $s^{-1}$ is $a^{-1}$ and $b \neq a$; the last letter of $s^{-1}$ is $b^{-1}$, the first letter of $u$ is $a$ and $b \neq a$.


\smallskip

Now let $u$ be not cyclically reduced, that is $a = b^{-1}$. 

If $w$ is cyclically reduced, then by applying the previous argument to $w, u$ we have that there exists a word $s$ of length at most 1 such that $w s u s^{-1}$ is cyclically reduced. Therefore the cyclic permutation $u s^{-1} w s$ of the former word is cyclically reduced and we have proved the claim.

Now let $w$ be not cyclically reduced; this implies that $d = c^{-1}$, that is $w = c w_1 = w_2 c^{-1}$. If $a \neq c$, then $u w$ is cyclically reduced. Let $a = c$; since $X$ has at least two elements, there exists $x \in X \cup X^{-1}$ such that $x \neq a, a^{-1}$. Then by taking $s = x$ we have that $u s w s^{-1}$ is cyclically reduced.
\end{proof}

\begin{corollary} \label{ax=b2} Let $u, w$ be reduced non-empty words. Then there exist infinitely many pairs of cyclically reduced words $v, v'$ such that $v$ is a cyclic permutation of $v'$ and $u*v = v'*u = \hat{\rho}(w)$.
\end{corollary} 

\begin{proof} By Lemma \ref{ax=b}, since $u^{-1}$ and $w$ are reduced non-empty words, there exists a word $s$ such that for every $n \in \mathbb{N}\setminus\{0\}$ the word $u^{-1} s^n w s^{-n}$ is cyclically reduced. Let us fix an $n$ and let us set $v := u^{-1} s^n w s^{-n}$; we have that 
	$$u * v = u * (u^{-1} s^n w s^{-n}) = \hat{\rho}(u u^{-1} s^n w s^{-n}) = \hat{\rho}(s^n w s^{-n}).$$ 

Since $w$ is reduced, then by Remark \ref{scope} there exists a word $t$ such that $w = t \hat{\rho}(w) t^{-1}$. This implies that $\hat{\rho}(s^n w s^{-n}) = \hat{\rho}(s^n t \hat{\rho}(w) t^{-1} s^{-n})$.

Since $s^n t \hat{\rho}(w) t^{-1} s^{-n}$ is reduced, by Corollary \ref{permCycRedForm} we have that 	
	$$\hat{\rho}(s^n t \hat{\rho}(w) t^{-1} s^{-n}) = \hat{\rho}(w),$$
thus $u*v = \hat{\rho}(w)$.

Now let us set $v' := s^n w s^{-n} u^{-1}$. We have that $v'$ is cyclically reduced because it is a cyclic permutation of a cyclically reduced word. With the same reasoning as above we prove that $v'*u = \hat{\rho}(w)$.
\end{proof}

\textit{Address:}

Carmelo Vaccaro

Laboratoire de Mathématiques d'Orsay

Université Paris-Saclay 

Bâtiment 307, rue Michel Magat

91400 Orsay

\medskip

\textit{e-mail:} \textsf{carmelo.vaccaro@universite-paris-saclay.fr}


\begin{thebibliography}{99}
	

	




\bibitem{BMc} R. S. Bowman and S. B. McCaul, \textit{Fast searching for Andrews–Curtis trivializations}, Experiment. Math., \textbf{15}(2006), 193--198.
	
\bibitem{BH} R. Brown and J. Huebschmann, \textit{Identities among relations}, in: \textit{Low Dimentional Topology}, London Math. Soc. Lect. Notes, 48: 153-202. Cambridge Univ. Press, 1982.
	
\bibitem{ChofKar} C. Choffrut and J. Karhum\"{a}ki, \textit{Combinatorics of words}, in: G. Rozenberg and A. Salomaa (eds), \textit{Handbook of Formal Languages}, Springer, 1997.


\bibitem{EncyMath} M. Hazewinkel (ed.), \textit{Encyclopedia of Mathematics}, Springer Science/ Kluwer, 2001.
	
\bibitem{Ivanov06} S. V. Ivanov, \textit{On Rourke’s extension of group presentations and a cyclic version of the Andrews–Curtis conjecture}, Proc. Amer. Math. Soc. \textbf{134}(2006), 1561–-1567.
	
\bibitem{Ivanov18} S. V. Ivanov, \textit{On conjectures of Andrews and Curtis}, Proc. Amer. Math. Soc. \textbf{146}(2018),  2283–-2298.

\bibitem{Jordan} C. R. Jordan, D. A. Jordan, \textit{Groups}, Newnes, 2004.
	
\bibitem{Kar} J. Karhum\"{a}ki, \textit{Combinatorics of words}, available online at the author's webpage.

\bibitem{LS} R. C. Lyndon, P. E. Schupp, \textit{Combinatorial Group Theory}, Springer, 1977.
	

	
\bibitem{MKS} W. Magnus, A. Karrass, D. Solitar, \textit{Combinatorial Group Theory}, Cambridge University Press, 2003.

\bibitem{PU} D. Panteleev, A. Ushakov, \textit{Conjugacy search problem and the Andrews-Curtis conjecture}, Groups, Complexity, Cryptology \textbf{11}(2019), 43–-60.

\bibitem{Rourke} C.P. Rourke, \textit{Presentations and the trivial group}, in Topology of low-dimensional manifolds, Lecture Notes in Math., vol. 722, Springer-Verlag, 1979, pp. 134--143.

\bibitem{Scarabotti} F. Scarabotti, \textit{On the presentations of the trivial group}, J. Group Theory \textbf{2}(1999), 319--327.

\bibitem{Short} H. Short, \textit{Diagrams and Groups}, in The Geometry of the Word Problem for Finitely Generated Groups, Birkhäuser, 2007.

\bibitem{PhDThs} C. Vaccaro, \textit{Algorithmic and geometric methods for characterizing all the relators of a group presentation}, Ph.D. Thesis, University of Palermo, 2009.

\bibitem{TwAs1} C. Vaccaro, \textit{Twisted associativity of the cyclically reduced product of words, part 1}, arXiv:1909.04863, \textit{https://arxiv.org/abs/1909.04863}.

\bibitem{TwAs2} C. Vaccaro, \textit{Twisted associativity of the cyclically reduced product of words, part 2}, arXiv:1910.09300, \textit{https://arxiv.org/abs/1910.09300}.



	
\end{thebibliography}
\end{document}